\documentclass[11pt,twoside]{article}

\usepackage{fancyhdr}
\usepackage{latexsym,amsmath,amssymb,amsfonts}
\usepackage{epsfig,graphics,color}
\usepackage{hyperref}
\usepackage{rotating}
\usepackage{pdfsync}
\usepackage{graphics}

\usepackage{amsmath, amsthm, amscd, amsfonts, graphicx,makeidx,amssymb,mathrsfs, fancyhdr,float,capt-of}
\usepackage[all]{xy}
\usepackage{verbatim}
\usepackage[mathcal]{euscript}
\usepackage{graphicx,epsfig}
\usepackage{pifont}
\usepackage{color}
\definecolor{rojo}{rgb}{1,0,0} 
\xyoption{arc}

\RequirePackage[mathcal,mathbf]{euler}

\newtheorem{thm}{Theorem}[section]
\newtheorem{cor}[thm]{Corollary}
\newtheorem{lem}[thm]{Lemma}
\newtheorem{prop}[thm]{Proposition}
\newtheorem{example}{Example}[section]
\theoremstyle{definition}
\newtheorem{defn}[thm]{\textbf{Definition}}

\theoremstyle{definition}

\theoremstyle{remark}
\newtheorem{rem}[thm]{Remark}
\renewcommand {\k} {\Bbbk}

\newcommand {\C} {\mathbb{C}}
\newcommand{\ben}{\begin{equation}}
\newcommand{\een}{\end{equation}}

\newcommand{\complex}{\ensuremath{{\mathbb C}}}

\newcommand{\op}{\operatorname}

\newcommand{\al}{\alpha}
\newcommand{\be}{\beta}
\newcommand{\ga}{\gamma}

\newcommand{\de}{\Delta}

\newcommand{\rt}{\rightarrow}

\newcommand{\ma}{\mathcal}
\newcommand{\mb}{\mathbb}

\newdir{ >}{{}*!/-10pt/@{>}}

\newcommand{\var}{\varepsilon}
\newcommand{\A}{\ma{A}}
\newcommand{\ot}{\otimes}

\newcommand{\h}{\op{H}}

\begin{document}

\title{Constructing nearly Frobenius algebras}
\author{ \large{ Dalia Artenstein\thanks{IMERL, Facultad de Ingenier\'ia, Montevideo, Uruguay}, Ana Gonz\'alez\thanks{IMERL, Facultad de Ingenier\'ia, Montevideo, Uruguay} and Marcelo Lanzilotta\thanks{IMERL, Facultad de Ingenier\'ia, Montevideo, Uruguay}}}

\date{}

\maketitle

\medskip
\renewcommand{\abstractname}{Abstract}
\begin{abstract}
In the first part we study nearly Frobenius algebras. The concept of
nearly Frobenius algebras is a generalization of the concept of
Frobenius algebras. Nearly Frobenius algebras do not have traces,
nor they are self-dual. We prove that the known constructions:
direct sums, tensor, quotient of nearly Frobenius algebras admit natural nearly Frobenius structures.\\
 In the second part we study algebras associated to some families of quivers and the nearly Frobenius
structures that they admit. As a main theorem, we prove that an
indecomposable algebra associated to a bound quiver $(Q,I)$ with no
monomial relations admits a non trivial nearly Frobenius structure
if and only if the quiver is $\overrightarrow{\mb{A}_n}$ and $I=0$.
We also present an algorithm that determines the number of
independent nearly Frobenius structures for Gentle algebras without
oriented cycles.
\end{abstract}
\section*{Introduction}
\hspace{.5cm}A Frobenius algebra over a field $\k$ is a
(non-necessarily commutative) associative algebra $A$, together with
a non-degenerate trace $\var: A \rightarrow \k$. In other words we
have that $\langle a,b\rangle =\var(ab)$ is a non-degenerate
bilinear form. They have been studied since the 1930's, specially in
representation theory, for their very nice duality properties.
 In recent times the surprising connection found to topological quantum field theories has made them subject of renewed interest.\\
An important example for us, of Frobenius algebra, is the Poincar\'e
algebra associated to every compact closed manifold $M$,
$A=\h^*(M)$. In this case we can define the trace as $\var(w)
=\int_M w,$ for $w\in \h^*(M)$. It is a classical result that
Poincar\'e duality is equivalent to the assertion that this trace is
non-degenerate. In topology this fact manifests in many ways, for
instance in the existence of an intersection product in homology
that becomes a coproduct in cohomology. The coproduct $\de$ is the
composition of the Poincar\'e duality isomorphism
$\xymatrix{D:\h_*(M)\ar[r]^{\cong}&\h^*(M)}$ with the dual map for
the ordinary cup product $\mu:\h^*(M)\ot\h^*(M)\rightarrow\h^*(M)$.
Note that if we consider the case of a non-compact manifold $M$, its
cohomology algebra is no longer a Frobenius algebra, but we may ask
ourselves what structure remains. In this way we arrive at the
following definition. \\A \emph{nearly Frobenius algebra} $A$ is an
algebra together with a coassociative coproduct $\de:A\rightarrow
A\ot A$ such that $\de$ is an $A$-bimodule morphism. To the best of
our knowledge this notion was first isolated by R. Cohen and V.
Godin, see \cite{cohen}. A good reference for this topic is the book
\cite{AnaErnesto}.

The other objects studied in this work are the quivers and the
associated algebras. It is a known result that  to each finite
dimensional basic algebra over an algebraically closed field $\k$
corresponds a graphical structure, called a \emph{quiver}, and that,
conversely, to each quiver corresponds an associative $\k$-algebra,
which has an identity and is finite dimensional under some
conditions. Similarly, using the quiver associated to an algebra
$A$, it will be possible to visualice a (finitely generated)
$A$-module as a family of (finite dimensional) $\k$-vector spaces
connected by linear maps (see \cite{ASS06}). The idea of such a
graphical representation seems to go back to the late forties but it
became widespread in the early seventies, mainly due to Gabriel
\cite{Gabriel72}, \cite{Gabriel73}. In an explicit form, the notions
of quiver and linear representation of quiver were introduced by
Gabriel in \cite{Gabriel72}. It was the starting point of the modern
representation theory of associative algebras.

In section 1 we present known concepts required in the rest of the
work. We dedicate section 2 to develop the concept of nearly
Frobenius algebras. Studying nearly Frobenius structures over an
algebra we prove that this family defines a $\k$-vector space. This
result permit us to define the Frobenius dimension of an algebra as
the dimension of this vector space. Moreover, in this section we
determine the Frobenius dimension of particular algebras as the
matrix algebra, the group algebra and the truncated polynomial
algebra. All these cases verify that $\op{Frobdim}\bigl(A\bigr)\leq
\op{dim}_\k\bigl(A\bigr)$. In section 3 we show that known
constructions, as opposite algebra, direct sum, tensor product and
quotient of nearly Frobenius algebras admit natural nearly Frobenius
structures. The last section is divided in three parts. In the first
part we prove that an indecomposable algebra associated to a bound
quiver $\bigl(Q,I\bigr)$ with no monomial relations  admits a non
trivial nearly Frobenius structure if and only if the quiver is
$\overrightarrow{\mb{A}_n}$ and $I=0$. Moreover, in this case the
Frobenius dimension is one. In the second part we deal with gentle
algebras. If the quiver associated to a gentle algebra $A$ has no
oriented cycles we show that the Frobenius dimension of $A$ is
finite and  we determine this number by an algorithm. In the last
part we exhibit a family of algebras $A=\bigl\{A_C\bigr\}_C$ given
by bound quivers for which
$\op{Frobdim}\bigl(A_C\bigr)>\op{dim}_\k\bigl(A_C\bigr)$.

\section{Preliminaries}
\begin{defn}
A \emph{quiver} $Q = \bigl(Q_0,Q_1, s, t\bigr)$ is a quadruple
consisting of two sets: $Q_0$ (whose elements are called
\emph{points}, or \emph{vertices}) and $Q_1$ (whose elements are
called \emph{arrows}), and two maps $s, t : Q_1 \rt Q_0$ which
associate to each arrow $\al\in Q_1$ its \emph{source} $s(\al)\in
Q_0$ and its \emph{target} $t(\al)\in Q_0$, respectively.
\end{defn}
An arrow $\al\in Q_1$ of source $a = s(\al)$ and target $b = t(\al)$
is usually denoted by $\al: a\rt b$. A quiver $Q = \bigl(Q_0,Q_1, s,
t\bigr)$ is usually denoted briefly by $Q = (Q_0,Q_1)$ or even
simply by $Q$. Thus, a quiver is nothing but an oriented graph
without any restriction on the number of arrows between two points,
to the existence of loops or oriented cycles.

\begin{defn}
Let $Q = \bigl(Q_0,Q_1, s, t\bigr)$ be a quiver and $a, b\in Q_0$. A
\emph{path} of \emph{length} $l\geq 1$ with source $a$ and target
$b$ (or, more briefly, from $a$ to $b$) is a sequence
$$\bigl(a |\al_1,\al_2,\dots,\al_l| b\bigr),$$
where $\al_k\in Q_1$ for all $1\leq k\leq l$,
$s\bigl(\al_1\bigr)=a$, $t\bigl(\al_k\bigr)=s\bigl(\al_{k+1}\bigr)$
for each $1\leq k<l$, and $t\bigl(\al_l\bigr)=b$. Such a path is
denoted briefly by $\al_1\al_2\dots\al_l$.
\end{defn}

\begin{defn}
Let $Q$ be a quiver. The \emph{path algebra} $\k Q$ is the
$\k$-algebra whose underlying $\k$-vector space has as its basis the
set of all paths $\bigl(a|\al_1,\al_2,\dots,\al_l|b\bigr)$ of length
$l\geq 0$ in $Q$ and such that the product of two basis vectors
$\bigl(a|\al_1,\al_2,\dots,\al_l|b\bigr)$ and $\bigl(c
|\be_1,\be_2,\dots,\be_k|d\bigr)$ of $\k Q$ is defined by
$$\bigl(a|\al_1,\al_2,\dots,\al_l|b\bigr)\bigl(c |\be_1,\be_2,\dots,\be_k|d\bigr)=\delta_{bc}\bigl(a|\al_1,\dots,\al_l,\be_1,\dots,\be_k|d),$$
where $\delta_{bc}$ denotes the Kronecker delta. In other words, the
product of two paths $\al_1\dots\al_l$ and $\be_1\dots\be_k$ is
equal to zero if $t\bigl(\al_l\bigr) \neq s\bigl(\be_1\bigr)$ and is
equal to the composed path $\al_1\dots\al_l\be_1\dots\be_k$ if
$t\bigl(\al_l\bigr) = s\bigl(\be_1\bigr)$. The product of basis
elements is then extended to arbitrary elements of $\k Q$ by
distributivity.
\end{defn}

Assume, that $Q$ is a quiver and $\k$ is a field. Let $\k Q$ be the associated path algebra. Denote by $R_Q$ the two-sided ideal in $\k Q$ generated by all paths of length 1, i.e. all arrows. This ideal is known as the arrow ideal.\\
It is easy to see, that for any $m\geq 1$ we have that $R^m_Q$ is a
two-sided ideal generated by all paths of length $m$. Note, that we
have the following chain of ideals:
$$R^2_Q\supseteq R^3_Q \supseteq R^4_Q\supseteq\cdots$$
\begin{defn}
A two-sided ideal $I$ in $\k Q$ is said to be \emph{admissible} if
there exists $m\geq 2$ such that
$$R^m_Q\subseteq I \subseteq R^2_Q.$$
\end{defn}

\begin{defn}
Let $\k$ be a field, and $Q$ a quiver. We call a finite dimensional
$\k$-algebra $A$ \emph{gentle} if it is Morita equivalent to an
algebra $\displaystyle{\frac{\k Q}{I}}$ where $Q$ is a quiver and
$I\subset \k Q$ an admissible ideal subject to the following
conditions:
\begin{enumerate}
\item[(1)] \begin{itemize}
             \item at each vertex of $Q$ at most 2 arrows start,
             \item at each vertex of $Q$ at most 2 arrows finish;
           \end{itemize}
\item[(2)] \begin{itemize}
             \item for each arrow $\be\in Q_1$ there is at most one arrow $\ga\in Q_1$ with $\be\ga$ a path not contained in $I$,
             \item for each arrow $\be\in Q_1$ there is at most one arrow $\al\in Q_1$ with $\al\be$ a path not contained in $I$;
           \end{itemize}
\item[(3)] the ideal $I$ is generated by paths of length 2;
\item[(4)] \begin{itemize}
             \item for each arrow $\be\in Q_1$ there is at most one arrow $\ga'\in Q_1$ with $\be\ga'$ a path contained in $I$,
             \item for each arrow $\be\in Q_1$ there is at most one arrow $\al'\in Q_1$ with $\al'\be$ a path contained in $I$.
           \end{itemize}
\end{enumerate}
\end{defn}

\section{Nearly Frobenius algebras}
\hspace{.5cm}The concept of nearly Frobenius algebras is a
generalization of the concept of Frobenius algebras. Nearly
Frobenius algebras do not have traces, nor they are self-dual.

\begin{defn}
A $\k$-algebra $A$ is a \emph{nearly-Frobenius algebra} if  there
exists a linear map $\de:A\rt A\otimes A$ such that
\begin{enumerate}
\item $\de$ is coassociative
\begin{equation*}
\begin{split}
\xymatrix@R=1cm@C=1.5cm{
A\ar[r]^{\de}\ar[d]_{\de}& A\otimes A\ar[d]^{\de\otimes 1}\\
A\otimes A\ar[r]_{1\otimes \de}&A\otimes A\otimes A}
\end{split}
\end{equation*}
\item $\de$ is a morphism of $A$-bimodule
\begin{equation*}
\begin{split}
\xymatrix@R=1cm@C=1.5cm{
A\otimes A\ar[r]^{m}\ar[d]_{\de\otimes 1}&A\ar[d]^{\de}\\
A\otimes A\otimes A\ar[r]_{1\otimes m} &A\otimes A}\quad \xymatrix{
A\otimes A\ar[d]_{1\otimes \de}\ar[r]^{m}&A\ar[d]^{\de}\\
A\otimes A\otimes A\ar[r]_{m\otimes 1}&A\otimes A }
\end{split}
\end{equation*}
\end{enumerate}
\end{defn}

\begin{rem}\label{lemma1}
Any nearly Frobenius coproduct in a $\k$-algebra is determined by
the evaluation in the unit of the algebra structure, that is
if $A$ is a $\k$-algebra and $\de:A\rt A\ot A$ is a $\k$-linear map such that $$\de(x)=\bigl(x\ot 1\bigr)\de(1)=\de(1)\bigl(1\ot x\bigr)$$ for all $x\in A$. 
\end{rem}
\begin{thm}\label{teo1}
Let $A$ be a fixed $\k$-algebra and $\ma{E}$ the set of nearly
Frobenius coproducts of $A$ making it into a nearly Frobenius
algebra. Then $\ma{E}$ is a $\k$-vector space.
\end{thm}
\begin{proof}
To prove that $\ma{E}$ is a $\k$-vector space we prove that $\ma{E}$
is a subspace of $V=\bigl\{\de:A\rt A\ot A\;\mbox{linear
transformation}\bigr\}$, which is a $\k$-vector space. We consider
the linear map $\de=\al\de_1+\be\de_2:A\rightarrow A\otimes A$, with
$\al,\be\in\k$ where $\de_1, \de_2\in\ma{E}$. First we prove that
this map is an $A$-bimodule morphism, i.e. $\bigl(m\otimes
1\bigr)\bigl(1\otimes\de\bigr)=\de\circ m=\bigl(1\otimes
m\bigl)\bigl(\de\otimes 1\bigr)$
$$\begin{array}{rcl}
     \bigl(m\otimes 1\bigr)\bigl(1\otimes\de\bigr) & = & \bigl(m\otimes 1\bigr)\bigl(1\otimes(\al\de_1+\be\de_2)\bigr) \\
      & = & \al\bigl(m\otimes 1\bigr)\bigl(1\otimes\de_1\bigr)+\be\bigl(m\otimes 1\bigr)\bigl(1\otimes\de_2\bigr) \\
      & = & \al\de_1m+\be\de_2m=\de m.
  \end{array}
$$
To prove the coassociativity of $\de$: $\bigl(\de\otimes
1\bigr)\de=\bigl(1\otimes\de\bigr)\de$, we fix a basis, as
$\k$-vector space, of $A$, $\ma{B}=\{e_i\}_{i\in I}$ and we note, by
the Remark \ref{lemma1}, that
\begin{equation}\label{equ1}
\de\bigl(e_k\bigr)=\bigl(e_k\otimes
1\bigr)\de(1)=\de(1)\bigl(1\otimes e_k\bigr).
\end{equation}
If we represent $e_ie_j=\sum_ka_{ij}^ke_k$,
$\de_1(1)=\sum_{i,j}b_{ij}e_i\ot e_j$ and
$\de_2(1)=\sum_{i,j}c_{ij}e_i\ot e_j$, then the equation
(\ref{equ1}) by $\de_1$ and $\de_2$ is expressed as:
$$\sum_{i,j,l}b_{ij}a_{ki}^le_l\ot e_j=\sum_{i,j,l}b_{li}a_{ik}^je_l\ot e_j,$$
$$\sum_{i,j,l}c_{ij}a_{ki}^le_l\ot e_j=\sum_{i,j,l}c_{li}a_{ik}^je_l\ot e_j,$$
therefore
\begin{equation}\label{equ2}
 \sum_ib_{ij}a_{ki}^l=\sum_ib_{li}a_{ik}^j,
\end{equation}
\begin{equation}\label{equ3}
 \sum_ic_{ij}a_{ki}^l=\sum_ic_{li}a_{ik}^j.
\end{equation}

Using the definition of $\de$, the coassociativity condition is
equivalent to
\begin{equation}\label{equ4}
\bigl(\bigl(\de_2\otimes
1\bigr)\de_1-\bigl(1\otimes\de_1\bigr)\de_2\bigr)+\bigl(\bigl(\de_1\otimes
1 \bigr)\de_2-\bigl(1\otimes\de_2\bigr)\de_1\bigr)=0.
\end{equation}

We will prove that $\bigl(\de_1\otimes
1\bigr)\de_2-\bigl(1\otimes\de_2\bigr)\de_1=0=\bigl(\de_2\otimes
1\bigr)\de_1-\bigl(1\otimes\de_1\bigr)\de_2$. To prove that the map
$\bigl(\de_1\otimes 1\bigr)\de_2-\bigl(1\otimes\de_2\bigr)\de_1$ is
zero is enough to observe that the evaluation in $1$ is zero:
$$\begin{array}{rcl}
    \bigl(\de_1\otimes 1\bigr)\de_2(x)-\bigl(1\otimes\de_2\bigr)\de_1(x) & = & \bigl(\de_1\otimes 1\bigr)\de_2(1)\bigl(1\otimes
x\bigr)-\bigl(1\otimes\de_2\bigr)\de_1(1)\bigl(1\otimes x\bigr) \\
     & = & \bigl(\bigl(\de_1\otimes
     1\bigr)\de_2(1)-\bigl(1\otimes\de_2\bigr)\de_1(1)\bigr)\bigl(1\otimes x\bigr)=0.
  \end{array}
$$
And the last equation holds from
$$\bigl(\de_1\otimes 1\bigr)\de_2(1)=\sum_{i,j}c_{ij}\de_1(e_i)\ot e_j=\sum_{j,l,m}\left(\sum_{i,k}c_{ij}b_{kl}a_{ik}^m\right)e_m\otimes e_l\otimes e_j$$
and $$\begin{array}{rcl}
      \bigl(1\otimes
\de_2\bigr)\de_1(1) & = &
\displaystyle{\sum_{i,j}b_{ij}e_i\ot\de_2(e_j)=\sum_{i,l,m}\left(\sum_{j,k}b_{ij}c_{kl}a_{jk}^m\right)e_i\otimes
e_m\otimes e_l} \\
       & = &
       \displaystyle{\sum_{j,l,m}\left(\sum_{i,k}b_{mk}c_{ij}a_{ki}^l\right)e_m\ot
       e_l\ot e_j}\\
       & = &
       \displaystyle{\sum_{j,l,m}\left(\sum_{i}c_{ij}\left(\sum_kb_{mk}a_{ki}^l\right)\right)e_m\ot e_l\ot e_j}\\
       & = &
       \displaystyle{\sum_{j,l,m}\left(\sum_{i}c_{ij}\left(\sum_kb_{kl}a_{ik}^m\right)\right)e_m\ot e_l\ot e_j}\quad
       \mbox{using}\,(\ref{equ2}),\;\mbox{and}\;(\ref{equ3}).
    \end{array}
$$
\end{proof}

\begin{defn} The \emph{Frobenius space} associated to an algebra $A$ is the vector space of all the possible coproducts $\Delta$ that make it into a nearly Frobenius algebra ($\ma{E}$ from Theorem \ref{teo1}). Its dimension over $\k$ is called the \emph{Frobenius dimension} of $A$, that is,
$$\op{Frobdim\bigl(A\bigr)}=\op{dim}_\k\bigl(\ma{E}\bigr).$$
\end{defn}

\begin{example}
Every Frobenius algebra is also a nearly Frobenius algebra.
\end{example}

It is known that the truncated polynomial algebra is a Frobenius
algebra, in the next example we prove that this algebra admits
nearly Frobenius structures that do not come from Frobenius
structures.
\begin{example}\label{example1}
Let $A$ be the truncated polynomial algebra in one variable
$\k[x]/x^{n+1}$. We will determine all nearly Frobenius structures
on $A$, even more we will determine a basis of the Frobenius space
$\ma{E}$ of $A$.

We consider the canonical basis $B=\bigl\{1,x,\dots,x^n\bigr\}$ of
$A$.  Then the general expression of a $\k$-linear map $\de:A\rt
A\ot A$ in the value $1$ is
$$\de(1)=\sum_{i,j=1}^na_{ij}x^i\ot x^j.$$
This map is an $A$-bimodule morphism if
\begin{equation}\label{equation}\de\bigl(x^k\bigr)=\bigl(x^k\ot
1\bigr)\de(1)=\de(1)\bigl(1\ot
x^k\bigr),\quad\forall\;k\in\{0,\dots,n\}.\end{equation} The
equation (\ref{equation}) when $k=1$ is
$$\sum_{i,j=1}^na_{ij}x^{i+1}\ot x^j=\sum_{ij,=1}^na_{ij}x^i\ot x^{j+1}.$$
This happens if $a_{0j-1}=0$, $j=1,\dots,n$; $a_{i-10}=0$,
$i=1,\dots,n$ and $a_{ij-1}=a_{i-1j}$.  Then
$$\de(1)=\sum_{k=0}^na_{kn}\left(\sum_{i+j=n+k}x^i\ot x^j\right)$$
We denote $a_k=a_{kn}$. Applying the Remark \ref{lemma1} we need to
prove that $\de(x^k)=\bigl(x^k\ot 1\bigr)\de(1)=\de(1)\bigl(1\ot
x^k\bigr)$ to conclude  that $\de$ is an $A$-bimodule morphism.
$$\begin{array}{rcl}
    \de(1)\bigl(1\ot x^l\bigr) & = & \displaystyle{\sum_{k=0}^na_{k}\left(\sum_{i+j=n+k}x^i\ot x^j\right)\bigl(1\ot x^l\bigr)}
      =  \displaystyle{\sum_{k=0}^na_{k}\left(\sum_{i+j=n+k}x^i\ot x^{j+l}\right)} \\
     & = & \displaystyle{\sum_{k=0}^na_{k}\left(\sum_{i+m=n+k+l}x^i\ot x^m\right)}
      =  \displaystyle{\sum_{k=0}^na_{k}\left(\sum_{r+m=n+k}x^{r+l}\ot x^m\right)} \\
     & = & \displaystyle{\bigl(x^l\ot 1\bigr)\sum_{k=0}^na_{k}\left(\sum_{r+m=n+k}x^r\ot x^m\right)}
     =  \displaystyle{\bigl(x^l\ot 1\bigr)\de(1)}.
  \end{array}
$$
Finally, we need to check that this map is coassociative: Let
$x^l\in \ma{A}$ with $l\geq 0$.
$$
\begin{array}{rcl}
  \bigl(\de\ot 1\bigr)\bigl(\de\bigl(x^l\bigr)\bigr)  & = & \displaystyle{\bigl(\de\ot 1\bigr)\left(\sum_{k=0}^na_{k}\left(\sum_{i+j=n+k+l}x^i\ot x^j\right)\right)  = \sum_{k=0}^na_{k}\left(\sum_{i+j=n+k+l}\de\bigl(x^i\bigr)\ot x^j\right)} \\
   & = & \displaystyle{\sum_{k,m=0}^na_ka_m\left(\sum_{i+j=n+k+l}\sum_{r+s=n+m+i}x^r\ot x^s\ot x^j\right)} \\
   & = & \displaystyle{\sum_{k,m=0}^na_ka_m\left(\sum_{r+s+j=2n+m+k+l}x^r\ot x^s\ot x^j\right)} \\
  \bigl(1\ot\de\bigr)\bigl(\de\bigl(x^l\bigr)\bigr)  & = &  \displaystyle{\bigl(1\ot\de)\left(\sum_{k=0}^na_{k}\left(\sum_{i+j=n+k+l}x^i\ot x^j\right)\right)  = \sum_{k=0}^na_{k}\left(\sum_{i+j=n+k+l}x^i\ot \de\bigl(x^j\bigr)\right)}\\
  & = & \displaystyle{\sum_{k,m=0}^na_ka_m\left(\sum_{i+j=n+k+l}\sum_{r+s=n+m+j}x^i\ot x^r\ot x^s\right)} \\
  & = & \displaystyle{\sum_{k,m=0}^na_ka_m\left(\sum_{r+s+j=2n+m+k+l}x^r\ot x^s\ot x^j\right)}. \\
\end{array}
$$
Then the pair $\bigl(A,\de\bigr)$ is a nearly Frobenius algebra. In
particular we have that the coproduct $\de$ is a linear combination
of the coproducts $\de_k$ defined by
$$\de_k\bigl(x^l\bigr)=\sum_{i+j=n+k+l}x^i\ot x^j,\quad\mbox{for}\; k\in\{0,\dots,n\}$$
that is $\displaystyle{\de=\sum_{k=0}^n a_k\de_k}$ where $a_k\in\k$
for all $k\in\{1,\dots,n\}$. It is clear that the set of coproducts
$\de_k$ is a linearly independent set. Then
$$\ma{C}=\bigl\{\de_k:A\rt A\ot A, k\in\{0,1,\dots ,n\}\bigr\}$$
is a basis of $\ma{E}$, and $\op{Frobdim}(A)=n+1 (=\op{dim}_\k(A))$.

Note that $\de_0$ is the Frobenius coproduct of $A$ where the trace
map $\varepsilon:A\rt\k$ is given by
$\varepsilon\bigl(x^i\bigr)=\delta_{i,n}$ and it is the only
coproduct that admits a completion to Frobenius algebra structure.
This is because if we have a counit map $\varepsilon:A\rt\k$ then it
satisfies the counit axiom: $m(\varepsilon\ot
1)\bigl(\de_k\bigl(x^i\bigr)\bigr)=x^i$, $\forall i=0,1,\dots ,n$.
But
$$m(\varepsilon\ot 1)\bigl(\de_k\bigl(x^i\bigr)\bigr)=\sum_{j+l=n+k+i}\varepsilon\bigl(x^j\bigr)x^l$$ with $l>i$ so $m(\varepsilon\ot 1)\bigl(\de_k\bigl(x^i\bigr)\bigr)\neq x^i$ for $k\in\{1,\dots,n\}$.
\end{example}

\begin{example}
Let $A$ be the algebra  $\complex\bigl[\bigl[x,x^{-1}\bigr]\bigr]$
of formal Laurent series. Consider the coproducts given by:
$$\Delta_{j}\bigl(x^{i}\bigr)=\sum_{k+l=i+j} x^{k} \otimes x^{l}.$$
These coproducts define nearly Frobenius structures that do not come
from a Frobenius structure and $\op{Frobdim}(A)=\infty$.
\end{example}

\begin{example}
Let be $A$ the matrix algebra $M_{n\times n}(\k)$. We consider the
canonical basis of $A$, $\ma{B}=\bigl\{E_{ij}:\;i,j=1,\dots
,n\bigr\}$.

As in the example \ref{example1} we can prove that $M_{n\times
n}(\k)$ admits $n\times n$ independent coproducts,  these are
$\displaystyle{\de_{kl}\bigl(E_{ij}\bigr)=E_{ik}\ot E_{lj}}$, and a
general coproduct in $A$ is
$$\de\bigl(E_{ij}\bigr)=\sum_{k,l=1}^na_{kl}\de_{kl}\bigl(E_{ij}\bigr).$$
Then $\ma{C}=\bigl\{\de_{kl}: k,l\in\{1,\dots,n\}\bigr\}$ is a basis of $\ma{E}$ and $\op{Frobdim}(A)=n^2$.\\
The coproduct in the identity matrix is
$$\de(I)=\sum_{i=1}^n\de\bigl(E_{ii}\bigr)=\sum_{i=1}^n\sum_{k,l=1}^na_{kl}E_{ik}\otimes
E_{li}.$$ In the particular case that $a_{kl}=0$ if $k\neq l$ and
$a_{kk}=1$, for all $k\in\{1,\dots ,n\}$ we recover the Frobenius
coproduct
$$\de(I)=\sum_{i,k=1}^n E_{ik}\otimes E_{ki},$$ where the trace map $\var:M_{n\times n}(\k)\rightarrow\k$ is $\var(A)=\op{tr}(A)$.
\end{example}

\begin{example}\label{example2.5}
Let $G$ be a cyclic finite group. The group algebra $\k G$ is a
nearly Frobenius algebra. A basis of $\k G$ is
$\bigl\{g^i:\;i=1,\dots,n\bigr\}$ where $|G|=n$.

Using the bimodule condition of the coproduct we can prove that a
basis of the Frobenius space is
$$\ma{C}=\bigl\{\de_k:\k G\rt \k G\ot \k G: k\in{2,\dots, n}\bigr\}$$
where $\displaystyle{\de_k(1)=\sum_{i=1}^{k-1}g^i\ot
g^{k-i}+\sum_{i=k}^ng^i\ot g^{n+k-i}}$. Then we have that
$$\op{Frobdim}(A)=n-1 < \op{dim}(A)=n.$$

The general expression of any nearly Frobenius coproduct is
$$\de(1)=\sum_{k=2}^n a_k\left(\sum_{i=1}^{k-1}g^i\ot g^{k-i}+\sum_{i=k}^ng^i\ot g^{n+k-i}\right)$$
In the particular case that $a_i=0$, for $i\in\{2,\dots, n-1\}$ and
$a_n=1$ we have $$\de(1)=\sum_{k=1}^{n}g^k\ot g^{n-k}$$ the
Frobenius coproduct of the group algebra $A$ where the counit is
$\var\bigl(g^i\bigr)=\delta_{ni}$.
\end{example}

To complete the construction of the category of nearly Frobenius
algebras we need to define the morphisms of them.
\begin{defn}
Let $\bigl(A,\de_A\bigr)$ and $\bigl(B,\de_B\bigr)$ be nearly
Frobenius algebras. We say that $f:A\rt B$ is a \emph{morphism of
nearly Frobenius algebras} if it is a morphism of algebras and the
next diagram commutes
$$\xymatrix@C=1.5cm{A\ar[r]^{\de_A}\ar[d]_{f}& A\ot A\ar[d]^{f\ot f}\\
B\ar[r]_{\de_B}&B\ot B}.$$
\end{defn}


\section{Constructing nearly Frobenius structures}
\hspace{.5cm}In this section we show that known constructions, as
opposite algebra, direct sum, tensor product and quotient of nearly
Frobenius algebras admit natural nearly Frobenius structures. A
basic but important remark in this section is the following.  If $A$
and $B$ are isomorphic $\k$-algebras, such that $B$ is a nearly
Frobenius algebra, we can provide to $A$ with a nearly Frobenius
structure, where the coproduct is defined as
$$\de_A(a)=\bigl(\psi\ot \psi\bigr)\de_B\bigl(\varphi(a)\bigr),$$
with $\varphi:A\rt B$ and $\psi:B\rt A$ morphisms of $\k$-algebras
such that $\psi\circ \varphi=\op{Id}_A$ and
$\varphi\circ\psi=\op{Id}_B$.

\begin{thm}\label{theorem1}
\begin{enumerate}
  \item An algebra  $A$ is nearly Frobenius if and only if $A^{op}$ is a nearly Frobenius algebra.
  \item Let $A_1,\dots, A_n$ be  nearly Frobenius $\k$-algebras then $A=A_1\oplus\dots\oplus A_n$ is a nearly Frobenius $\k$-algebra.
  \item If $A$ and $B$ are nearly Frobenius $\k$-algebras, then $A\otimes_\k B$ also is.
\end{enumerate}
\end{thm}
\begin{proof}
\begin{enumerate}
  \item 
      We define the coproduct $\de^{op}:A^{op}\rt A^{op}\ot A^{op}$ as $\tau\circ\de$, where $\de$ is the coproduct in $A$ and $\tau$ is the twist, that is $\tau(a\ot b)=b\ot a$. It is clear that $\de^{op}$ is coassociative because $\de$ is coassociative. We need to check that $\de^{op}$ is morphism of $A^{op}$-bimodule.
      $$\de^{op}(a\ast b)  =  \de^{op}(ba)  =  \tau\bigl(\de(ba)\bigr)  =  \sum a_2\ot ba_1  =  \sum a_2\ot a_1\ast b  =  \bigl(1\ot\ast\bigr)\bigl(\de^{op}(a)\ot b\bigr)$$
      $$\de^{op}(a\ast b)  =  \de^{op}(ba) =  \tau\bigl(\de(ba)\bigr)  =  \sum b_2a\ot b_1=\sum a\ast b_2\ot b_1 =\bigl(\ast\ot 1\bigr)\bigl(a\ot\de^{op}(b)\bigr) $$
  \item Let $A_1,\dots, A_n$ be nearly Frobenius algebras with $\de_1\dots,\de_n$ are the associated coproducts. We consider the canonical injections
 $\displaystyle{q_i:A_i\rt \bigoplus_{i=1}^nA_i}$. By the universal property of the direct sum in $\op{Vect}_\k$, there exists a unique morphism $\de$ in $\op{Vect}_\k$ such that the diagram
  $$\xymatrix@C=1.3cm@R=1cm{
  A_j\ar[r]^{q_j}\ar[d]_{\de_j}&\displaystyle{\bigoplus_{i=1}^nA_i}\ar[d]^{\de}\\
  A_j\ot A_j\ar[r]_(.3){q_j\ot q_j}&\displaystyle{\bigl(\bigoplus_{i=1}^nA_i\bigr)\otimes\bigl(\bigoplus_{i=1}^nA_i\bigr)}
  }$$ commutes.

  The coassociativity is a consequence of the commutativity of the cube
$$\xymatrix@C=1cm@R=1cm{
& A\ot A\ot A& & A\ot A\ar[ll]_(.5){\de\ot 1} \\
A_i\ot A_i\ot A_i\ar[ur]^{q_i\ot q_i\ot q_i}& & A_i\ot A_i\ar[ll]_(.4){\de_i\ot 1}\ar[ur]^{q_i\ot q_i}&  \\
& A\ot A\ar@{->}'[u][uu]^(.3){1\ot\de} & & A\ar@{->}'[l][ll]_(.3){\de}\ar[uu]^{\de} \\
A_i\ot A_i\ar[uu]^{1\ot\de_i}\ar[ur]_{q_i\ot q_i}& &
A_i\ar[uu]^(.6){\de_i}\ar[ll]^{\de_i}\ar[ur]_{q_i}&
  }$$
To prove the Frobenius identities, first we note that the diagram
$$\xymatrix{
A\ar[r]^{p_i}\ar[d]_{\de}& A_i\ar[d]^{\de_i}\\
A\ot A\ar[r]_{p_i\ot p_i}& A_i\ot A_i }$$ commutes, where
$\displaystyle{p_j:\bigoplus_{i=1}^n A_i\rt A_j}$ is the canonical
projection. This implies that the next cube commutes.
 $$\xymatrix@C=1cm@R=1cm{
& A\ot A\ar[rr]^{m}\ar[dl]_{p_i\ot p_i}\ar@{->}'[d][dd]_(.6){\de\ot 1}& & A\ar[dd]_{\de}\ar[dl]^{p_i} \\
A_i\ot A_i\ar[rr]^(.6){m_i}\ar[dd]^{\de_i\ot 1}& & A_i\ar[dd]_(.4){\de_i}&  \\
& A\ot A\ot A\ar@{->}'[r][rr]^(.3){1\ot m}\ar[dl]_{p_i\ot p_i\ot p_i} & & A\ot A\ar[dl]_(.5){p_i\ot p_i} \\
A_i\ot A_i\ot A_i\ar[rr]^{1\ot m_i}& & A_i\ot A_i &
  }$$ Then $(A,\de)$ is a nearly Frobenius algebra.
\item We can define the coproduct $\de:A\otimes B\rt\bigl(A\ot B\bigr)\ot\bigl(A\ot B\bigr)$ as
$$\de=(1\ot\tau\ot 1)\circ\bigl(\de_1\ot\de_2\bigr),\quad\mbox{where}\; \tau\;\mbox{is the twist.}$$
Using that $\de_1$ and $\de_2$ are the coproducts of $A$ and $B$
respectively we conclude that $\bigl(A\ot B,\de\bigr)$ is a nearly
Frobenius algebra.
\end{enumerate}
\end{proof}

\begin{cor}
Let $G$ be a finite group. If $char(\k)$ does not divide the order
of $G$, then $\k G$ is a nearly Frobenius algebra.
\end{cor}
\begin{proof}
Applying Maschke's theorem we have that $\k G$ is semisimple, then
it is a direct sum of simple algebras $M_{n_i\times n_i}(\k)$.
Therefore, by the Theorem \ref{theorem1}, we conclude that $\k G$
is a nearly Frobenius algebra. 
\end{proof}

\begin{cor}
If $G$ is a finite abelian group. Then $\k G$ is a nearly Frobenius
algebra.
\end{cor}
\begin{proof}
If $G$ is a finite abelian group, then, by the fundamental theorem
of finite abelian groups, $G=G_1\oplus\dots\oplus G_p$, where $G_i$
is a finite cyclic group for $i\in\{1,\dots,p\}$. Therefore, the
group algebra $\k G$ of $G$ is isomorphic, as a $\k$-algebra, to $\k
G_1\ot\dots\ot\k G_p$. Finally, applying the Example
\ref{example2.5} and the Theorem \ref{theorem1} we conclude that $\k
G$ is a nearly Frobenius algebra.
\end{proof}

\begin{defn}
Let  $\bigl(A, \de\bigr)$ be a nearly Frobenius algebra. A linear
subspace $J$ in $A$ is called a \emph{nearly Frobenius ideal} if
\begin{enumerate}
\item[(a)]$J$ is an ideal of $A$ and
\item[(b)] $\de(J)\subset J\ot A + A\ot J$.
\end{enumerate}
\end{defn}
Note that, if $A$ is a bialgebra, i.e. we have a trace map
$\var:A\rt\k$, the additional condition $\var(J)=0$ implies that $J$
is a bi-ideal of $A$.
\begin{example}
Go back to Example \ref{example1}. We observe that the ideal
$J=\bigl\langle x\bigr\rangle$ is a nearly Frobenius ideal if we
consider the coproduct $\de_1$. Because
$$\de_1(x)=\sum_{i+j=n+2}x^i\ot x^j=x^2\ot x^n+x^3\ot x^{n-1}+\cdots +x^n\ot x^2\in J\ot J\subset J\ot A+A\ot J.$$
\end{example}

\begin{prop}\label{p1}
Let $\bigl(A,\de\bigr)$ be a nearly Frobenius algebra, $J$ a nearly
Frobenius ideal. Then $A/J$ admits a unique nearly Frobenius
structure such that $p:A\rightarrow A/J$ is a nearly Frobenius
morphism.
\end{prop}
\begin{proof}
Since $(p\otimes p)\Delta(J)\subset(p\otimes p)\bigl(J\otimes
A+A\otimes J\bigr)=0$, by the universal property of the quotient
vector space it follows that there exists a unique morphism of
vector spaces $$\overline{\Delta}:A/J\rightarrow A/J\otimes A/J$$
for which the diagram
 $$\xymatrix{A\ar[r]^p\ar[d]_{\Delta}&A/J\ar[d]^{\overline{\Delta}}\\
A \ar[r]_(.3){p\otimes p}& A/J\otimes A/J}$$ is commutative. This
map is defined by
$\displaystyle{\overline{\Delta}(\overline{a})=\sum
\overline{a_1}\otimes\overline{a_2}}$ where $\overline{a}=p(a)$ and
$\displaystyle{\de(a)=\sum a_1\ot a_2}$, i.e.
$\overline{\Delta}=(p\otimes p)\circ \Delta$.

The fact that   $\bigl(\overline{\Delta}\otimes
1\bigr)\overline{\Delta}(\overline{a})=\bigl(1\otimes\overline{\Delta}\bigr)\overline{\Delta}(\overline{a})=\sum\overline{a_1}\otimes\overline{a_2}\otimes\overline{a_3}$
follows immediately from the commutativity of the diagram.

The last step is to prove that the coproduct is a bimodule morphism:
$$\xymatrix{A/J\otimes A/J\ar[r]^{\overline{m}}\ar[d]_{\overline{\Delta}\otimes 1}&A/J\ar[d]^{\overline{\Delta}}\\
A/J\otimes A/J\otimes A/J\ar[r]_(.6){1\otimes\overline{m}}&
A/J\otimes A/J}\quad
\xymatrix{A/J\otimes A/J\ar[r]^{\overline{m}}\ar[d]_{1\otimes\overline{\Delta}}&A/J\ar[d]^{\overline{\Delta}}\\
A/J\otimes A/J\otimes A/J\ar[r]_(.6){\overline{m}\otimes 1}&
A/J\otimes A/J}$$
$$\begin{array}{rclcl}
  \displaystyle{\overline{\Delta}\overline{m}(\overline{a}\otimes\overline{b})} & = & \displaystyle{\overline{\Delta}(p(ab))=(p\otimes
p)\Delta(ab)} & = & \displaystyle{(p\otimes p)\bigl((1\otimes m)(\Delta\otimes 1)(a\otimes b)\bigr)}\\
  & = &  \displaystyle{(p\otimes p)
  \left(\sum a_1\otimes a_2b\right)} & = & \displaystyle{\sum\overline{a_1}\otimes\overline{a_2b}}\\
 \displaystyle{(1\otimes\overline{m})(\overline{\Delta}\otimes
1)(\overline{a}\otimes\overline{b})}  & = &
\displaystyle{(1\otimes\overline{m})\left(\sum\overline{a_1}\otimes\overline{a_2}\otimes\overline{b}\right)}
& = & \displaystyle{\sum\overline{a_1}\otimes\overline{a_2b}}.
\end{array}$$
\end{proof}

\begin{example}
Consider the linear quiver
$$\includegraphics{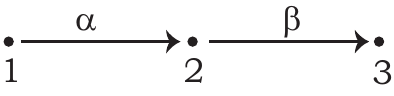}$$
and the associated path algebra
$$A=\bigl\langle e_1, e_2, e_3, \alpha, \beta, \alpha\beta\bigr\rangle$$

This algebra admits a unique nearly Frobenius coproduct:
$$\Delta(e_1)=\alpha\beta\otimes e_1,\quad \Delta(e_2)=\beta\otimes\alpha,\quad \Delta(e_3)= e_3\otimes \alpha\beta,$$
$$\Delta(\alpha)=\alpha\beta\otimes\alpha,\quad \Delta(\beta)=\beta\otimes\alpha\beta,\quad \Delta(\alpha\beta)=\alpha\beta\otimes\alpha\beta,$$
in particular $\op{Frobdim}(A)=1$.

Now, let be $J=\langle\al\be\rangle$. Note that $J$ is a nearly
Frobenius ideal:
$$\Delta(\alpha\beta)=\alpha\beta\otimes\alpha\beta\in A\ot
J+J\ot A.$$

Then, applying the Proposition \ref{p1}, $B=A/J=\bigl\langle
\overline{e_1}, \overline{e_2}, \overline{e_3}, \overline{\alpha},
\overline{\beta}\bigr\rangle$ admits a nearly Frobenius structure
defined by
$$\overline{\de}:B\rightarrow B\ot B$$
$$\overline{\de}\bigl(\overline{e_1}\bigr)=\overline{\de}\bigl(\overline{e_3}\bigr)=\overline{\de}\bigl(\overline{\al}\bigr)=\overline{\de}\bigl(\overline{\be}\bigr)=0\; \mbox{and}\; \overline{\de}\bigl(\overline{e_2}\bigr)=\overline{\be}\ot\overline{\al}.$$

Note that the algebra $B$ is associated to the quiver
$$\includegraphics{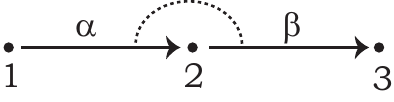}$$
where the dashed line represents the relation $\al\be=0$, admits
three independent coproducts, in fact $\op{Frobdim}(B)=3$:
$$\de_1\bigl(\overline{e_1}\bigr)=\overline{\alpha}\otimes \overline{e_1},\quad \de_1\bigl(\overline{e_2}\bigr)=\overline{e_2}\otimes
\overline{\alpha},\quad
\de_1\bigl(\overline{e_3}\bigr)=\de_1(\overline{\be})=0,\quad
\de_1\bigl(\overline{\al}\bigr)=\overline{\al}\ot\overline{\al}$$
$$\de_2\bigl(\overline{e_1}\bigr)=\de_2\bigl(\overline{\al}\bigr)=0,\quad \de_2\bigl(\overline{e_2}\bigr)=\overline{\be}\ot\overline{e_2},
\quad\de_2\bigl(\overline{e_3}\bigr)=\overline{e_3}\ot\overline{\be},\quad\de_2\bigl(\overline{\be}\bigr)=\overline{\be}\ot\overline{\be}$$
$$\de_3\bigl(\overline{e_1}\bigr)=\de_3\bigl(\overline{e_3}\bigr)=\de_3\bigl(\overline{\al}\bigr)=\de_3\bigl(\overline{\be}\bigr)=0,\quad \de_3\bigl(\overline{e_2}\bigr)=\overline{\be}\ot\overline{\al}.$$
Observe that $\overline{\de}$ coincides with $\de_3$.
\end{example}

\begin{thm}\label{teo1}
Let $A$, $B$ and $C$ nearly Frobenius algebras. Given two
epimorphisms of nearly Frobenius algebras $f_A:A\twoheadrightarrow
C$ and $f_B:B\twoheadrightarrow C$ the pullback $R$ of $f_A$ and
$f_B$
$$R=\bigl\{\bigl(a,b\bigr)\in A\times B: f_A\bigl(a\bigr)=f_B\bigl(b\bigr)\bigr\}$$ is a nearly Frobenius algebra.
\end{thm}
\begin{proof}
The pullback $R$ is a subalgebra of $A\times B$, then the product is
defined by $$\bigl(a,b\bigr)\cdot\bigl(c,d\bigr)=\bigl(ac,bd\bigr)$$
for all $\bigl(a,b\bigr),\bigl(c,d\bigr)\in R$. Note that
$\bigl(ac,bd\bigr)\in R$ because
$$f_A\bigl(ac\bigr)=f_A\bigl(a\bigr)f_A\bigl(c\bigr)=f_B\bigl(b\bigr)f_B\bigl(d\bigr)=f_B\bigl(bd\bigr).$$
As $A$ and $B$ are nearly Frobenius algebras there exist $\de_A:A\rt
A\ot A$ and $\de_B:B\rt B\ot B$ coproducts. Then, we define
$$\de_R:R\rt R\ot R$$ as $\de_R\bigl((a,b)\bigr)=\sum\bigl(a_1,b_1\bigr)\ot\bigl(a_2,b_2\bigr)$, where
$\de_A\bigl(a\bigr)=\sum a_1\ot a_2$ and $\de_B\bigl(b\bigr)=\sum
b_1\ot b_2.$

First, we check that the map $\de_R$ is well defined, that is
$\de_R\bigl((a,b)\bigr)\in R\ot R$ for $(a,b)\in R$. As the maps
$f_A$ and $f_B$ are morphisms of nearly Frobenius algebras the next
diagrams commute
$$\xymatrix@C=1.5cm{A\ar[r]^{\de_A}\ar[d]_{f_A}& A\ot A\ar[d]^{f_A\ot f_A}\\
C\ar[r]_{\de_C}&C\ot C}\quad
\xymatrix@C=1.5cm{B\ar[r]^{\de_B}\ar[d]_{f_B}&B\ot B\ar[d]^{f_B\ot
f_B}\\C\ar[r]_{\de_C}&C\ot C},$$ then $\bigl(f_A\ot
f_A\bigr)\de_A(a)=\de_C\bigl(f_A(a)\bigr)=\de_C\bigl(f_B(b)\bigr)=\bigl(f_B\ot
f_B\bigr)\de_B(b)$. Using this we have
$$\begin{array}{rcl}
   \bigl(f_A\ot f_A\bigr)\bigl(\sum a_1\ot a_2\bigr)  & = & \bigl(f_B\ot f_B\bigr)\bigl(\sum b_1\ot b_2\bigr) \\
     & \Rightarrow &  \\
    \sum f_A\bigl(a_1\bigr)\ot f_A\bigl(a_2\bigr) & = & \sum f_B\bigl(b_1\bigr)\ot f_B\bigl(b_2\bigr) \\
     & \Rightarrow &  \\
    \left(\sum f_A\bigl(a_1\bigr)\right)\ot\left(\sum f_A\bigl(a_2\bigr)\right) & = & \left(\sum f_B\bigl(b_1\bigr)\right)\ot\left(\sum f_B\bigl(b_2\bigr)\right).
  \end{array}
$$
Then $\sum f_A\bigl(a_1\bigr)=\sum f_B\bigl(b_1\bigr)$ and $\sum
f_A\bigl(a_2\bigr)=\sum f_B\bigl(b_2\bigr)$ and $$f_A\left(\sum
a_1\right)=f_B\left(\sum b_1\right)\quad\mbox{and}\quad
f_A\left(\sum a_2\right)=f_B\left(\sum b_2\right)$$ therefore
$$\de_R\bigl((a,b)\bigr)=\sum\bigl(a_1,b_1\bigr)\ot\sum\bigl(a_2,b_2\bigr)\in
R\ot R.$$
\begin{enumerate}
\item Coassociativity of $\de_ R$: $\bigl(\de_R\ot 1\bigr)\de_R\bigl((a,b)\bigr)=\bigl(1\ot\de_R\bigr)\de_R\bigl((a,b)\bigr)$.
$$\begin{array}{rcl}
    \bigl(\de_R\ot 1\bigr)\de_R\bigl((a,b)\bigr) & = & \sum\de_R\bigl(a_1,b_1\bigr)\ot\bigl(a_2,b_2\bigr) \\
     & = & \sum\sum\bigl(a_{11},b_{11}\bigr)\ot\bigl(a_{12},b_{12}\bigr)\ot\bigl(a_2,b_2\bigr) \\
    \bigl(1\ot\de_R\bigr)\de_R\bigl((a,b)\bigr) & = & \sum\bigl(a_1,b_1\bigr)\ot\de_R\bigl(a_2,b_2\bigr) \\
     & = & \sum\sum \bigl(a_1,b_1\bigr)\ot\bigl(a_{21},b_{21}\bigr)\ot\bigl(a_{22},b_{22}\bigr).
  \end{array}
$$ As the coproducts $\de_A$ and $\de_B$ are coassociatives the expressions $\sum\sum\bigl(a_{11},b_{11}\bigr)\ot\bigl(a_{12},b_{12}\bigr)\ot\bigl(a_2,b_2\bigr)$ and $\sum\sum \bigl(a_1,b_1\bigr)\ot\bigl(a_{21},b_{21}\bigr)\ot\bigl(a_{22},b_{22}\bigr)$ coincide, then the coproduct $\de_R$ is coassociative.
\item $\de_R$ is a morphism of bimodules if $$\xymatrix@C=1.5cm{R\ot R\ar[r]^{m}\ar[d]_{\de_R\ot 1}& R\ar[d]^{\de_R}\\
R\ot R\ot R\ar[r]_{1\ot m}&R\ot R}\quad \xymatrix@C=1.5cm{R\ot
R\ar[r]^{m}\ar[d]_{1\ot\de_R}&R\ar[d]^{\de_R}\\R\ot R\ot
R\ar[r]_{m\ot 1}&R\ot R}$$ commute. We will prove that the first
diagram commutes, the other case is analogous.

We knows that $\de_A$ and $\de_B$ are bimodule morphisms, then
\begin{equation}\label{eq1}
\sum a_1\ot a_2c=\sum (ac)_1\ot (ac)_2,
\end{equation}
for all $a, c\in A$, and
\begin{equation}\label{eq2}
\sum b_1\ot b_2d=\sum(bd)_1\ot(bd)_2,
\end{equation} for all $b, d\in B$.

Let $\bigl(a,b\bigr),\bigl(c,d\bigr)\in R$, then
$$\begin{array}{rcl}
                           \de_R\bigl(ac,bd\bigr) & = & \sum \bigl((ac)_1,(bd)_1\bigr)\ot\bigl((ac)_2,(bd)_2\bigr)  \\
                            \bigl(\de_R\ot 1\bigr)\bigl((a,b)\ot(c,d)\bigr)& = & \sum \bigl(a_1,b_1\bigr)\ot\bigl(a_2,b_2\bigr)\ot\bigl(c,d\bigr)   \\
                            & \Rightarrow &    \\
                           \bigl(1\ot m\bigr)\bigl(\de_R\ot 1\bigr)\bigl((a,b)\ot(c,d)\bigr) & = & \sum  \bigl(a_1,b_1\bigr)\ot \bigl(a_2c,b_2d\bigr)
                         \end{array}
$$
Using (\ref{eq1}) and (\ref{eq2}) we have that $\sum
\bigl((ac)_1,(bd)_1\bigr)\ot\bigl((ac)_2,(bd)_2\bigr)=\sum
\bigl(a_1,b_1\bigr)\ot \bigl(a_2c,b_2d\bigr).$ Then the first
diagram commutes.
\end{enumerate}
\end{proof}

\begin{example}
Let be the quivers $Q_A$, $Q_B$ and $Q_C$ illustrated in the next
picture,
$$\includegraphics{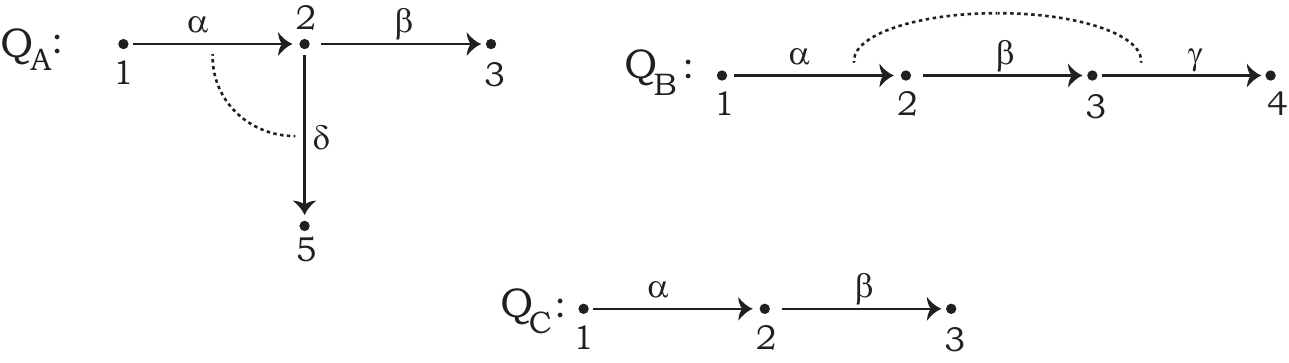}$$ as before the dashed lines represent the relations $\al\delta=0$ and $\al\be\ga=0$.
The pullback algebra $R=A\times_{C}B$ where $f_A:A\rt C$ and
$f_B:B\rt C$ are the natural projections, by the previous theorem,
admits a nearly Frobenius structure. In the next step we develop the
associated coproduct.

First, we describe the nearly Frobenius structures of the algebras $A$, $B$ and $C$.\\

The path algebra $C$ admits only one independent coproduct, this is
$$\begin{array}{rclcrclrcll}
    \de\bigl(e_1\bigr) & = & \al\be\ot e_1, &  &  \de\bigl(\al\bigr) & = & \al\be\ot\al,  &  & \de\bigl(e_2\bigr)     & = & \be\ot\al, \\
     \de\bigl(\be\bigr) & = & \be\ot\al\be, &  &  \de\bigl(e_3\bigr) & = &  e_3\ot\al\be, &  &  \de\bigl(\al\be\bigr) & = & \al\be\ot\al\be.
  \end{array}
$$
The path algebra $B$ admits three independent coproducts, these are
$$\begin{array}{rclrcl}
    \de\bigl(e_1\bigr) & = & a\al\be\ot e_1,   &  \de\bigl(\al\bigr) & = & a\al\be\ot\al, \\
    \de\bigl(e_2\bigr) & = & a\be\ot\al+b\be\ga\ot e_2+c\be\ga\ot\al,   & \de\bigl(e_4\bigr) & = & be_4\ot\be\ga,  \\
    \de\bigl(e_3\bigr) & = & ae_3\ot\al\be+b\ga\ot\be+c\ga\ot\al\be,   &  \de\bigl(\ga\bigr) & = & b\ga\ot\be\ga,\\
    \de\bigl(\be\bigr) & = & a\be\ot\al\be+b\be\ga\ot\be+c\be\ga\ot\al\be,   &  \de\bigl(\al\be\bigr) & = & a\al\be\ot\al\be,\\
                       &   &                     &   \de\bigl(\be\ga\bigr)& = & b\be\ga\ot\be\ga.
  \end{array}
$$
The path algebra $A$ admits only one independent coproduct, this is
$$\begin{array}{rclrclrclrcl}
    \de\bigl(e_1\bigr) & = & \al\be\ot e_1, &  \de\bigl(\al\bigr)    & = & \al\be\ot\al, &  \de\bigl(e_2\bigr) & = & \be\ot\al, &  \de\bigl(\be\bigr) & = & \be\ot\al\be,\\
    \de\bigl(e_3\bigr) & = & e_3\ot\al\be,    & \de\bigl(\delta\bigr) & = & 0,             & \de\bigl(e_5\bigr) & = & 0,          & \de\bigl(\al\be\bigr) & = & \al\be\ot\al\be.
  \end{array}
$$
The pullback algebra $R$ is defined by the next diagram
$$\xymatrix{
R\ar[r]^{\pi_1}\ar[d]_{\pi_2}& A\ar[d]^{f_A}\\
B\ar[r]_{f_B}&C }$$ Then $$
\begin{array}{c}
  R=\left\langle \bigl(e_1,e_1\bigr),\bigl(e_2,e_2\bigr),\bigl(e_3,e_3\bigr),\bigl(\al,\al,\bigr),\bigl(\be,\be\bigr),\bigl(\al\be,\al\be\bigr),\bigl(e_5,0\bigr),
\bigl(\delta,0\bigr),\bigl(0,e_4\bigr),\bigl(0,\ga\bigr),\bigl(0,\be\ga\bigr)\right\rangle \\
  \simeq \\
  \bigl\langle e_1, e_2, e_3, \al, \be, \al\be, e_5, \delta, e_4, \ga, \be\ga \bigr\rangle,
\end{array}
$$
that is the path algebra associated to the pushout quiver
$Q_A\coprod_{Q_C}Q_B$.

Finally, the coproduct of $R$, by the last identification, is
$$\begin{array}{rclrclrclrcl}
    \de\bigl(e_1\bigr) & = & \al\be\ot e_1,  & \de\bigl(\al\bigr)     & = & \al\be\ot\al & \de\bigl(e_2\bigr) & = & \be\ot\al, &  \de\bigl(\be\bigr) & = & \be\ot\al\be, \\
    \de\bigl(e_3\bigr) & = &  e_3\ot\al\be,  & \de\bigl(\gamma\bigr)  & = & 0,            & \de\bigl(e_4\bigr) & = & 0,          &  \de\bigl(\delta\bigr) & = & 0, \\
    \de\bigl(e_5\bigr) & = & 0,               & \de\bigl(\al\be\bigr) & = & \al\be\ot\al\be, &    \de\bigl(\be\ga\bigr) & = & 0. &&&
  \end{array}$$

\end{example}
Using the Lemma 2.1.2 of \cite{Jessica} we have that $R$ is the path
algebra of the pushout quiver $Q_R=Q_A\coprod_{Q_C}Q_B$. This quiver
is represented in the next picture.
$$\includegraphics{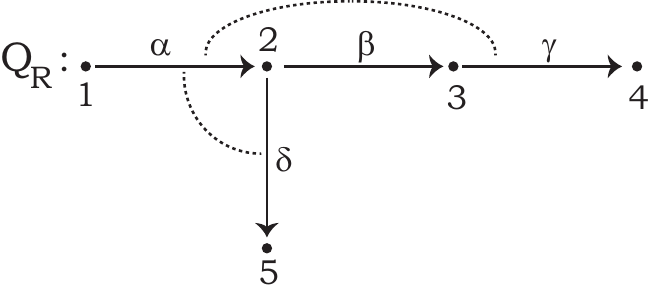}$$
The algebra associated to the pushout quiver $Q_A\coprod_{Q_C}Q_B$
is generated by $$\bigl\{e_1, e_2, e_3, e_4, e_5, \al, \be, \ga,
\delta, \al\be, \be\ga\bigr\}.$$ This algebra admits two independent
nearly Frobenius coproducts, these are
$$\begin{array}{rclcrcl}
    \de\bigl(e_1\bigr) & = & a\al\be\ot e_1 &  & \de\bigl(\al\bigr)  & = & a\al\be\ot\al \\
    \de\bigl(e_2\bigr) & = &a\be\ot\al+b\be\ga\ot\al &  & \de\bigl(\be\bigr)       & = & \be\ot\al\be+b\be\ga\ot\al\be \\
    \de\bigl(e_3\bigr) & = & a e_3\ot\al\be+b\ga\ot\al\be &  & \de\bigl(\gamma\bigr)    & = & 0 \\
    \de\bigl(e_4\bigr) & = & 0 &  & \de\bigl(\delta\bigr)    & = & 0 \\
    \de\bigl(e_5\bigr) & = & 0 &  & \de\bigl(\al\be\bigr)    & = & a\al\be\ot\al\be \\
     &  &                      &  & \de\bigl(\be\gamma\bigr) & = & 0
  \end{array}
$$
Note that if $b=0$ we have the coproduct detected by the pullback
structure defined in the Theorem \ref{teo1} and developed in the
previous example.

\section{Quivers and nearly Frobenius algebras}
\hspace{.5cm} This section is divided in three parts. In the first part we prove that an indecomposable algebra associated to a bound quiver $\bigl(Q,I\bigr)$
with no monomial relations  admits a non trivial nearly Frobenius structure if and only if the quiver is $\overrightarrow{\mb{A}_n}$ and $I=0$. Moreover, in this case the Frobenius dimension is one. In the second part we deal with gentle algebras. If the quiver associated to a gentle algebra $A$ has no oriented cycles we show that the Frobenius dimension of $A$ is finite and  we determine this number by an algorithm. In the last part we exhibit a family of algebras $A=\bigl\{A_C\bigr\}_C$ given by bound quivers for which
$\op{Frobdim}\bigl(A_C\bigr)>\op{dim}_\k\bigl(A_C\bigr)$.

\subsection{Path algebras}
\begin{lem}\label{e2}
If $Q=\overrightarrow{\mb{A}_n}$, that is, $Q$ is the following
quiver
$$\includegraphics{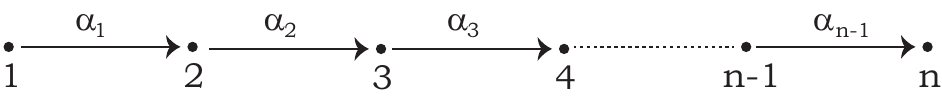},$$
the path algebra $A=\k Q$,  $$\k Q=\bigl\langle e_1,
e_2,\dots,e_n,\al_i\dots \al_{j}:i=1,\dots,n, i\leq j\leq n
\bigr\rangle,$$ admits only one independent nearly Frobenius
structure, where the coproduct is defined as follows
$$\begin{array}{rcl}
    \de\bigl(e_1\bigr)&=&\al_1\dots\al_{n-1}\ot e_1; \\
    \de\bigl(e_n\bigr)&=& e_n\ot\al_1\dots\al_{n-1}; \\
    \de\bigl(e_i\bigr)&=&\al_i\dots\al_{n-1}\ot\al_1\dots\al_{i-1},\quad i=2,\dots ,n-1; \\
    \de\bigl(\al_i\dots\al_j\bigr)&=&\al_i\dots\al_{n-1}\ot\al_1\dots\al_j,\quad 1\leq i\leq j\leq n-1.
  \end{array}
$$
\end{lem}
\begin{proof}
If we have a coproduct $\de$ the next condition is required
\begin{equation}\label{equat1}
\de\bigl(e_i\bigr)=\de\bigl(e_i\bigr)\bigl(1\ot
e_i\bigr)=\bigl(e_i\ot 1\bigr)\de\bigl(e_i\bigr),\quad\forall
i=1,\dots, n.
\end{equation}
This implies that the coproduct in the vertexes $e_1$ and $e_n$ is
$$\begin{array}{rcl}
    \de\bigl(e_1\bigr) & = & \displaystyle{a_0^1e_1\ot e_1+\sum_{i=1}^{n-1}a_i^1\al_1\dots\al_i\ot e_1},\quad a_i^1\in\k \\
    \de\bigl(e_n\bigr) & = & \displaystyle{a_0^ne_n\ot e_n+\sum_{i=1}^{n-1}a_i^ne_n\ot\al_i\dots\al_{n-1}},\quad a_i^n\in\k.
  \end{array}
$$
As
$\displaystyle{\de\bigl(\al_1\dots\al_{n-1}\bigr)=\de\bigl(e_1\bigr)\bigl(\al_1\dots\al_{n-1}\bigr)=\bigl(\al_1\dots\al_{n-1}\bigr)\de\bigl(e_n\bigr)}$
we have that
$$a_i^1=a_i^n=0\;\forall i=0,\dots,n-2.$$
Then the coproduct in these vertexes is
$$\begin{array}{ccc}
    \displaystyle{\de\bigl(e_1\bigr)} & = & \displaystyle{a\al_1\dots\al_{n-1}\ot e_1}, \\
    \displaystyle{\de\bigl(e_n\bigr)} & = & \displaystyle{a e_n\ot\al_1\dots\al_{n-1}
    }.
  \end{array}
$$
Using the equation (\ref{equat1}) the coproduct in the vertex $e_i$,
$i=2,\dots,n-1$ is
$$\de\bigl(e_i\bigr)=a_0^ie_i\ot e_i+\sum_{j=1}^{i-1}a_j^ie_i\ot\al_j\dots\al_{i-1}+\sum_{j=i}^{n-1}a_j^i\al_i\dots\al_j\ot e_i+\sum_{j=1}^{i-1}\sum_{k=i}^{n-1}a_{jk}^i\al_i\dots\al_k\ot\al_j\dots\al_{i-1}.$$
The coproduct in the path $\al_1\dots\al_{i-1}$ is given by
$$\de\bigl(\al_1\dots\al_{i-1}\bigr)=\de\bigl(e_1\bigr)\bigl(1\ot\al_1\dots\al_{i-1}\bigr)=\bigl(\al_1\dots\al_{i-1}\ot 1\bigr)\de\bigl(e_i\bigr),$$
then
$$\begin{array}{rcl}
    \de\bigl(\al_1\dots\al_{i-1}\bigr) & = & \displaystyle{a\al_1\dots\al_{n-1}\ot\al_1\dots\al_{i-1}} \\
     & = & \displaystyle{a_0^i\al_1\dots\al_{i-1}\ot e_i+\sum_{j=1}^{i-1}a_j^i\al_1\dots\al_{i-1}\ot\al_j\dots\al_{i-1}} \\
     & + & \displaystyle{\sum_{j=i}^{n-1}a_j^i\al_1\dots\al_{i-1}\al_i\dots\al_j\ot e_i}\\ & + & \displaystyle{\sum_{j=1}^{i-1}\sum_{k=i}^{n-1}a_{jk}^i\al_1\dots\al_{i-1}\al_i\dots\al_k\ot\al_j\dots\al_{i-1}}
  \end{array}
$$
therefore  $a_0^i=a_j^i=0$, $\forall j=1,\dots n-1,$ $a_{jk}^i=0$,
$\forall j=2,\dots n-1, k=1,\dots n-2$, $a_{1 n-1}^i=a$ and
$\displaystyle{\de\bigl(e_i\bigr)=a\al_i\dots\al_{n-1}\ot\al_1\dots\al_{i-1}}$.
Also, this determine the coproduct on paths $\al_i\dots\al_j$:
$$\de\bigl(\al_i\dots\al_j\bigr)=a\al_i\dots\al_{n-1}\ot\al_1\dots\al_j.$$
To conclude the construction we need to check that $\de$ is
coassociative.
$$\begin{array}{rcl}
    (\de\ot 1)\de\bigl(e_i\bigr) & = & \displaystyle{(\de\ot 1)\bigl(a\al_i\dots\al_{n-1}\ot\al_1\dots\al_{i-1}\bigr)} \\
                                 & = & \displaystyle{a^2\al_i\dots\al_{n-1}\ot\al_1\dots\al_{n-1}\ot\al_1\dots\al_{i-1}}, \\
    (1\ot\de)\de\bigl(e_i\bigr)  & = & \displaystyle{(1\ot \de)\bigl(a\al_i\dots\al_{n-1}\ot\al_1\dots\al_{i-1}\bigr)} \\
                                 & = & \displaystyle{a^2\al_i\dots\al_{n-1}\ot\al_1\dots\al_{n-1}\ot\al_1\dots\al_{i-1}}.
  \end{array}
$$
$$\begin{array}{rcl}
    (\de\ot 1)\de\bigl(\al_i\dots\al_j\bigr) & = & \displaystyle{(\de\ot 1)\bigl(a\al_i\dots\al_{n-1}\ot\al_1\dots\al_j\bigr)} \\
                                             & = & \displaystyle{a^2\al_i\dots\al_{n-1}\ot\al_1\dots\al_{n-1}\ot\al_1\dots\al_{j}},\\
    (1\ot \de)\de\bigl(\al_i\dots\al_j\bigr) & = & \displaystyle{(\de\ot 1)\bigl(a\al_i\dots\al_{n-1}\ot\al_1\dots\al_j\bigr)} \\
                                             & = & \displaystyle{a^2\al_i\dots\al_{n-1}\ot\al_1\dots\al_{n-1}\ot\al_1\dots\al_{j}}.
  \end{array}
$$
Then, a basis of the Frobenius space has only one coproduct and
$\op{Frobdim}(A)=1$.
\end{proof}

\begin{lem}\label{lemma2}
Let $\displaystyle{A=\frac{\k Q}{I}}$ be a finite dimensional
algebra.
If $\alpha, \mu\in Q_{1}$ with
$s\bigl(\alpha\bigr)=s\bigl(\mu\bigr)=p$
($t\bigl(\alpha\bigr)=t\bigl(\mu\bigr)=p$) such that no monomial
relation ends (starts) on $\al$
 or $\be$. Then $\Delta\bigl(e_p\bigr)=\Delta\bigl(\alpha\bigr)=\Delta\bigl(\mu\bigr)=0$
for all nearly Frobenius structure $\Delta$.

\end{lem}
\begin{proof}
We prove the first case, the other is analogous. The situation is
the following
$$\includegraphics[scale=1.0]{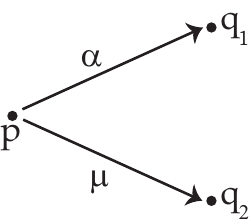}$$ with $p,
q_{1}$ and $q_{2}\in Q_{0}$, $\alpha, \mu\in Q_{1}$,
$s\bigl(\alpha\bigr)=s\bigl(\mu\bigr)=p$,
$t\bigl(\alpha\bigr)=q_{1}$ and $t\bigl(\mu\bigr)=q_{2}$. Since
$\Delta\bigl(e_p\bigr)=\bigl(e_p\ot
1\bigr)\Delta\bigl(e_p\bigr)=\Delta\bigl(e_p\bigr)\bigl(1\ot
e_p\bigr)$ the elementary tensors appearing in
$\Delta\bigl(e_{p}\bigr)$ must have the first coordinate starting in
$e_{p}$ and the second coordinate ending in $e_{p}$, this means:
$$\sum a_{ij}\al_i\ot\be_j$$
where $s\bigl(\al_i\bigr)=p$ and $t\bigl(\be_j\bigr)=p$.\\ In the
same way we have that: $$\de\bigl(e_{q_1}\bigr)=\sum
b_{ij}\ga_i\ot\delta_j$$ with $s\bigl(\ga_i\bigr)=q_1$,
$t\bigl(\delta_j\bigr)=q_1$.

Since
$\Delta(\alpha)=\Delta\bigl(e_p\bigr)\bigr(1\ot\alpha\bigr)=\bigl(\al\ot
1\bigr)\Delta\bigl(e_{q_{1}}\bigr)$ we have that
$$\de\bigl(\al\bigr)=\sum b_{ij}\al\ga_i\ot\delta_j=\sum a_{ij}\al_i\ot\be_j\al,$$
then $\al\ga_i=\al_i$ and $\delta_j=\be_j\al$ (some
$\alpha\gamma_{i}$ could be zero but $\be_j\al\neq 0$ $\forall j$
since there is no relation ending on $\al$). Therefore
$$\de\bigl(e_p\bigr)=\sum a_{ij}\al\ga_i\ot\be_j.$$
On the other hand we have that $$\de\bigl(e_{q_2}\bigr)=\sum
c_{ij}\eta_i\ot\xi_j,$$ with $s\bigl(\eta_i\bigr)=q_2$,
$t\bigl(\xi_j\bigr)=q_2$. Then
$$\de(\mu)=\sum c_{ij}\mu\eta_i\ot\xi_j=\sum a_{ij}\al_i\ot\be_j\mu.$$
So we conclude that $\de\bigl(e_p\bigr)=\sum a_{ij}\mu\eta_i\ot\be_j.$\\
Comparing $\de\bigl(e_p\bigr)=\sum a_{ij}\mu\eta_i\ot\be_j$ and
$\de\bigl(e_p\bigr)=\sum a_{ij}\al\ga_i\ot\be_j$ we deduce that
$\de\bigl(e_p\bigr)=0$ and therefore
$\de\bigl(\al\bigr)=\de\bigl(\mu\bigr)=0$.
\end{proof}

\begin{thm}\label{theorem}
Let $\displaystyle{A=\frac{\k Q}{I}}$ be a finite dimensional
indecomposable algebra such that there are no monomial relations. If
$A$ admits a non trivial  nearly Frobenius structure then  $I=0$ and
$Q=\overrightarrow{\mb{A}_n}$.
\end{thm}
\begin{proof}
If there exist $p, q_{1}, q_{2}\in Q_{0}$ and $\alpha, \mu\in Q_{1}$
such that $s\bigl(\alpha\bigr)=s\bigl(\mu\bigr)=p$,
$t\bigl(\alpha\bigr)=q_{1}$ and $t\bigl(\mu\bigr)=q_{2}$, that is
$$\includegraphics{f17.pdf}$$
applying the Lema \ref{lemma2} we have that
$\de\bigl(e_p\bigr)=\de\bigl(\al\bigr)=\de\bigl(\mu\bigr)=0$.
Moreover, using that there exist no monomial relations, arguing in
the same way that in Lemma
\ref{lemma2} we conclude that $\de\bigl(e_{q_1}\bigr)=\de\bigl(e_{q_2}\bigr)=0$.\\
Since $A$ is a finite dimensional indecomposable algebra $Q$ is
finite and connected. Given a point $r$ of the quiver there is a
walk $w=p\sim p_1\sim \dots \sim p_s\sim r$ from $p$ to $r$, where
$\sim$ means that there is an arrow $p_i\rt p_j$ or $p_j\rt p_i$.
Then, since $\Delta\bigl(e_{p}\bigr)=0$ and there exist non monomial
relations, we can reproduce again the arguments and prove that
$\Delta\bigl(e_{p_1}\bigr)=
\dots=\Delta\bigl(e_{p_s}\bigr)=\Delta\bigl(e_{r}\bigr)=0$.
Therefore $\Delta\bigl(e_r\bigr)=0$ for any point of $Q_0$. Then the
coproduct $\de$ is trivial.
\end{proof}

\begin{cor}
Let $A$ be the path algebra associated to $Q$, a finite connected
quiver. Then $A$ admits a non trivial nearly Frobenius structure if
and only if $Q=\overrightarrow{\mb{A}_{n}}$.
\end{cor}
\begin{proof}
If $Q=\overrightarrow{\mb{A}_{n}}$, by the Lemma \ref{e2}, there
exists a unique non trivial nearly Frobenius structure on $A$.

Suppose now that $A$ is the path algebra associated to $Q$ with a
non trivial nearly Frobenius structure, then, by the Theorem
\ref{theorem}, we have that $Q=\overrightarrow{\mb{A}_{n}}$.
\end{proof}

\subsection{Gentle algebras}
\begin{lem}
The algebra associated to the quiver
$$\includegraphics{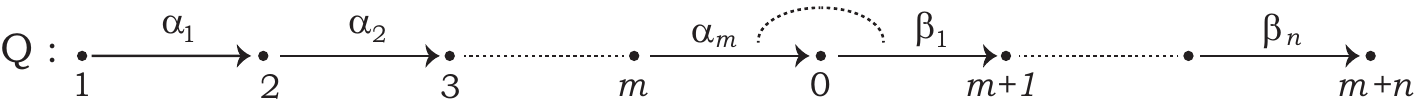}$$
with the relation $\al_m\be_1=0$, admits $mn+2$ independent nearly
Frobenius structures, these are
$$\begin{array}{rclcrcl}
\displaystyle{\de\bigl(e_1\bigr)} & = & \displaystyle{a\al_1\dots\al_m\ot e_1,} &  & \displaystyle{\de\bigl(e_{m+1}\bigr)} & = & \displaystyle{b\be_2\dots\be_n\ot\be_1,} \\
\displaystyle{\de\bigl(e_i\bigr)} & = & \displaystyle{a\al_i\dots\al_m\ot\al_1\dots\al_{i-1},} &  & \displaystyle{\de\bigl(e_{m+i}\bigr)} & = & \displaystyle{b\be_{i+1}\dots\be_n\ot\be_1\dots\be_i,} \\
\displaystyle{\de\bigl(e_m\bigr)} & = & \displaystyle{a\al_m\ot\al_1\dots\al_{m-1},} &  & \displaystyle{\de\bigl(e_{m+n}\bigr)} & = & \displaystyle{be_{m+n}\ot\be_1\dots\be_n,}\\
\displaystyle{\de\bigl(\al_i\dots\al_j\bigr)} & = &
\displaystyle{a\al_i\dots\al_m\ot\al_1\dots\al_j,}  &  &
\displaystyle{\de\bigl(\be_i\dots\be_j\bigr)} & = &
\displaystyle{b\be_i\dots\be_n\ot\be_1\dots\be_j,}
\end{array}$$
$$\de\bigl(e_0\bigr)=ae_0\ot\al_1\dots\al_m+b\be_1\dots\be_n\ot e_0+\sum_{i=1}^m\sum_{j=1}^nc_{ij}\be_1\dots\be_j\ot\al_i\dots\al_m,$$
where $a,b, c_{ij}\in\k$, $i=1,\dots ,m$ and $j=1,\dots ,n$.
Therefore $\op{Frobdim}\left(\frac{\k Q}{I}\right)=mn+2$.
\end{lem}
\begin{proof}
Using that the coproduct satisfies
$$\de\bigl(e_i\bigr)=\bigl(e_i\ot 1\bigr)\de\bigl(e_i\bigr)=\de\bigl(e_i\bigr)\bigl(1\ot e_i\bigr)$$
we conclude that, on the vertex $e_i$ for $i=1,\dots ,m$, the
coproduct is
$$\de\bigl(e_i\bigr)=a_ie_i\ot e_i+\sum_{j=i}^ma_i^j\al_i\dots\al_j\ot e_i+\sum_{j=1}^ia_j^ie_i\ot\al_j\dots\al_i+\sum_{j=i}^m\sum_{k=1}^ia_{jk}^i\al_i\dots\al_j\ot\al_k\dots\al_i.$$
and $$\de\bigl(e_1\bigr)=a_1e_1\ot
e_1+\sum_{j=1}^ma_1^j\al_1\dots\al_j\ot e_1.$$ Similarly, we have
that the coproduct on the vertex $e_{i+m}$, for $i=1,\dots ,n$ is
$$\de\bigl(e_{m+i}\bigr)=b_ie_{m+i}\ot e_{m+i}+\sum_{j=i}^nb_i^j\be_i\dots\be_j\ot e_{m+i}+\sum_{j=1}^ib_j^ie_{m+i}\ot\be_j\dots\be_i+\sum_{j=i}^n\sum_{k=1}^ib_{jk}^i\be_i\dots\be_j\ot\be_k\dots\be_i$$
and $$\de\bigl(e_{m+n}\bigr)=b_ne_{m+n}\ot
e_{m+n}+\sum_{j=1}^nb_j^ne_{m+n}\ot\be_j\dots\be_n.$$ In the vertex
$e_0$ the situation is different.
$$\de\bigl(e_0\bigr)=c_0e_0\ot e_0+\sum_{k=1}^mc^ke_0\ot\al_k\dots\al_m+\sum_{j=1}^nc_j\be_1\dots\be_j\ot e_0+\sum_{j=1}^n\sum_{k=1}^mc_{jk}\be_1\dots\be_k\ot\al_k\dots\al_m.$$
If we consider the path $\al_1\dots\al_m$ the coproduct
$\de\bigl(\al_1\dots\al_m\bigr)$ is
$$\de\bigl(e_1\bigr)\bigl(1\ot\al_1\dots\al_m\bigr)=\bigl(\al_1\dots\al_m\ot 1\bigr)\de\bigl(e_0\bigr).$$
Then $a_1=a_1^j=0$ for all $j=1,\dots m-1$ and $c_0=c^k=0$ for all
$k=2,\dots m$. The only coefficient not zero is $c^1=a_1^m$.

Therefore
$$\de\bigl(e_1\bigr)=a\al_1\dots\al_m\ot e_1,$$
$$\de\bigl(e_0\bigr)=ae_0\ot\al_1\dots\al_m+\sum_{j=1}^nc_j\be_1\dots\be_j\ot e_0+\sum_{j=1}^n\sum_{k=1}^mc_{jk}\be_1\dots\be_k\ot\al_k\dots\al_m.$$
On the other hand, if we consider the path $\be_1\dots\be_n$
$$\de\bigl(\be_1\dots\be_n\bigr)=\de\bigl(e_0\bigr)\bigl(1\ot\be_1\dots\be_n\bigr)=\bigl(\be_1\dots\be_n\ot 1\bigr)\de\bigl(e_{m+n}\bigr),$$
then $b_n=b_j^n=0$ for all $j=2,\dots n$ and $c_j=0$ for all
$j=1,\dots n-1$. The only coefficient not zero is $c_n=b_1^n$.

Consequently
$$\de\bigl(e_{m+n}\bigr)=b\be_1\dots\be_n\ot e_{m+n},$$
$$\de\bigl(e_0\bigr)=ae_0\ot\al_1\dots\al_m+b\be_1\dots\be_n\ot e_{0}+\sum_{j=1}^n\sum_{k=1}^mc_{jk}\be_1\dots\be_j\ot\al_k\dots\al_m.$$

To complete the prove we consider the internal paths
$\al_1\dots\al_{i-1}$, $\be_1\dots\be_i$, $\al_i\dots \al_j$ and
$\be_i\dots\be_j$.

To the first family of paths we have that
$$\de\bigl(\al_1\dots\al_{i-1}\bigr)=\de\bigl(e_1\bigr)\bigl(1\ot\al_1\dots\al_{i-1}\bigr)=\bigl(\al_1\dots\al_{i-1}\ot 1\bigr)\de\bigl(e_i\bigr),$$
then $a_{jk}^i=0$ for all $j=i,\dots,m-1, k=2,\dots ,i$,
$a_i=a_i^j=a_j^i=0$ for all $j=1,\dots, m$ and $a=a^i_{m1}$.
Accordingly
$$\de\bigl(e_i\bigr)=a\al_i\dots\al_m\ot\al_1\dots\al_{i-1},\quad\mbox{for
all}\; i=2,\dots, m.$$ In a very similar way we have that
$$\de\bigl(e_{m+i}\bigr)=b\be_{i+1}\dots\be_n\ot\be_1\dots\be_{i},\quad\forall\;i=1,\dots,
n-1.$$ An immediate consequence of these results is
$$\de\bigl(\al_i\dots\al_j\bigr)=a\al_i\dots\al_m\ot\al_1\dots\al_j\quad\mbox{and}\quad \de\bigl(\be_i\dots\be_j\bigr)=b\be_i\dots\be_n\ot\be_1\dots\be_j.$$
The coassociativity of the coproduct in the vertices $e_i,\,
e_{i+m}$ and in the arrows $\al_i\dots\al_j$, $\be_i\dots\be_j$ is
analogous to the example \ref{e2}. In the vertex $e_0$ is a simple
calculus.
\end{proof}

\begin{lem}\label{lemma1}
The algebra associated to the quiver
$$\scalebox{.8}{\includegraphics{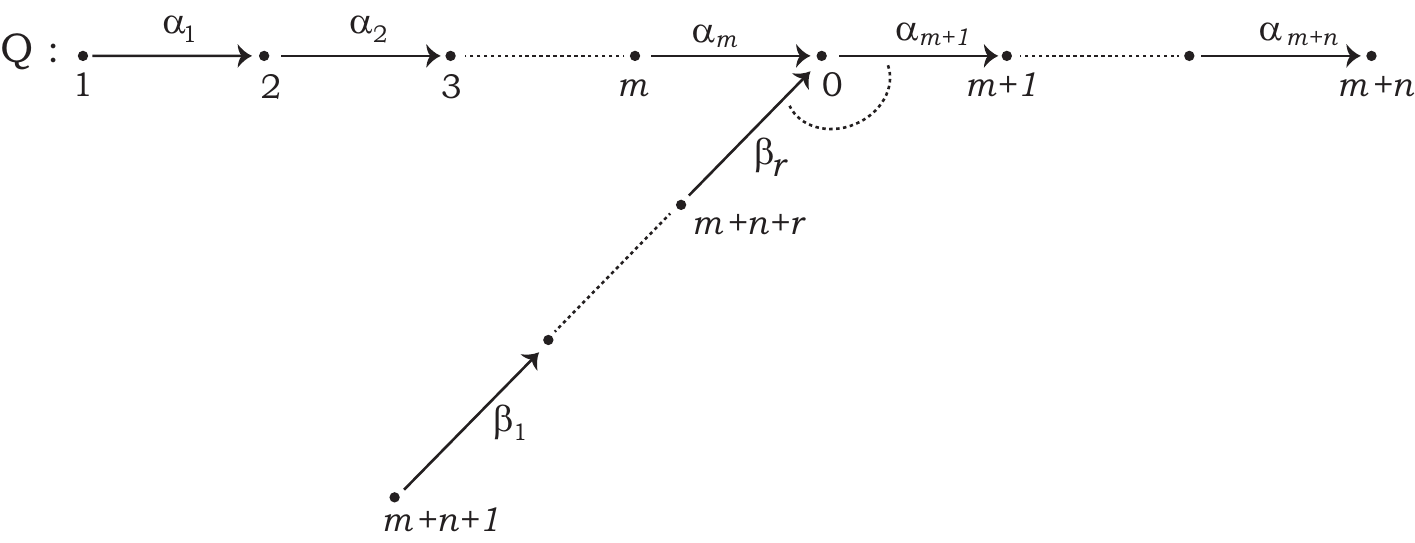}}$$
with the relation $\be_r\al_{m+1}=0$, admits only one nearly
Frobenius structure, this is
$$\begin{array}{rclcrcl}
\displaystyle{\de\bigl(e_1\bigr)} & = & \displaystyle{\al_1\dots\al_{m+n}\ot e_1,} &  & \displaystyle{\de\bigl(\al_1\bigr)} & = & \displaystyle{\al_1\dots\al_{m+n}\ot\al_1,}\\
\displaystyle{\de\bigl(e_i\bigr)} & = & \displaystyle{\al_i\dots\al_{m+n}\ot\al_1\dots\al_{i-1},} &  & \displaystyle{\de\bigl(\al_i\bigr)} & = & \displaystyle{\al_i\dots\al_{m+n}\ot\al_1\dots\al_i,} \\
\displaystyle{\de\bigl(e_{m+n}\bigr)} & = & \displaystyle{e_{m+n}\ot\al_1\dots\al_{m+n},} &  & \displaystyle{\de\bigl(\al_{m+n}\bigr)} & = & \displaystyle{\al_{m+n}\ot\al_1\dots\al_{m+n},} \\
  \end{array}$$
$$\de\bigl(e_{m+n+i}\bigr) =  0 \quad\mbox{and}\quad \de\bigl(\be_i\bigr)  =  0\quad\forall\; i=1,\dots,r.$$
Therefore $\op{Frobdim}(A)=1$.
\end{lem}
\begin{proof}
In the extreme vertices the coproduct has the general expression
$$\begin{array}{rcl}
    \de\bigl(e_1\bigr) & = &\displaystyle{ a_1e_1\ot e_1+\sum_{i=1}^{m+n}a_i^1\al_1\dots\al_i\ot e_1}, \\
    \de\bigl(e_{m+n}\bigr) & = & \displaystyle{b_ne_{m+n}\ot e_{m+n}+\sum_{i=1}^{m+n}b_i^ne_{m+n}\ot\al_i\dots\al_{m+n}
    }.
  \end{array}
$$
The coproduct in the path $\al_1\dots\al_{m+n}$ is
$$\de\bigl(\al_1\dots\al_{m+n}\bigr)=\de\bigl(e_1\bigr)\bigl(1\ot\al_1\dots\al_{m+n}\bigr)=\bigl(\al_1\dots\al_{m+n}\ot 1\bigr)\de\bigl(e_{m+n}\bigr),$$
then\\
$\displaystyle{a_1e_1\ot\al_1\dots\al_{m+n}+\sum_{i=1}^{m+n}a_i^1\al_1\dots\al_i\ot\al_1\dots\al_{m+n}=b_n\al_1\dots\al_{m+n}\ot
e_{m+n}}$\\
$\displaystyle{+\sum_{i=1}^{m+n}b_i^n\al_1\dots\al_{m+n}\ot\al_i\dots\al_{m+n}.}$

Consequently $a_i^1=b_j^n=0$, for all $i=1,\dots,m+n-1$, for all
$j=2,\dots,m+n$ and $a_{m+n}^1=b_1^{n}$,

and
$$\begin{array}{rcl}
    \de\bigl(e_1\bigr) & = &\displaystyle{ a\al_1\dots\al_{m+n}\ot e_1}, \\
    \de\bigl(e_{m+n}\bigr) & = & \displaystyle{ae_{m+n}\ot\al_1\dots\al_{m+n}
    }.
  \end{array}
$$
Repeating the procedure of the previous lemma we can prove that
$$\de\bigl(e_i\bigr)=a\al_i\dots\al_{m+n}\ot\al_1\dots\al_{i-1},\quad\forall\;
i=\dots, m+n.$$ In the vertex $e_0$ the situation is different
$$\begin{array}{rcl}
    \de\bigl(e_0\bigr) & = & \displaystyle{\de\bigl(e_0\bigr)\bigl(1\ot e_0\bigr)=\bigl(e_0\ot 1\bigr)\de\bigl(e_0\bigr)} \\
      & = & \displaystyle{c_0e_0\ot e_0+\sum_{j=1}^nc_j\al_{m+1}\dots\al_{m+j}\ot e_0 +\sum_{j=1}^mc^je_0\ot\al_k\dots\al_m}\\
      & + & \displaystyle{\sum_{j=1}^n\sum_{k=1}^mc_{jk}\al_{m+1}\dots\al_{m+j}\ot\al_k\dots\al_m }\\
      & + & \displaystyle{\sum_{j=1}^rd_je_0\ot\be_j\dots\be_r+\sum_{j=1}^n\sum_{k=1}^rd_{jk}\al_{m+1}\dots\al_{m+j}\ot\be_j\dots\be_r. }
  \end{array}
$$
By the other hand $$\de\bigl(e_{m+n+1}\bigr)=be_{m+n+1}\ot
e_{m+n+1}+\sum_{i=1}^rb_i\be_1\dots\be_i\ot e_{m+n+1}.$$ Now,
consider the path $\be_1\dots\be_r$, then the coproduct is described
by
$$\de\bigl(\be_1\dots\be_r\bigr)=\de\bigl(e_{m+n+1}\bigr)\bigl(1\ot\be_1\dots\be_r\bigr)=\bigl(\be_1\dots\be_r\ot 1\bigr)\de\bigl(e_0\bigr),$$
then
$$\begin{array}{rr}
\displaystyle{ be_{m+n+1}\ot\be_1\dots\be_r+\sum_{i=1}^rb_i\be_1\dots\be_i\ot\be_1\dots\be_r = }& \displaystyle{ \sum_{j=1}^mc^j\be_1\dots\be_r\ot\al_k\dots\al_m}\\
 &\displaystyle{+\sum_{j=1}^rd_j\be_1\dots\be_r\ot\be_j\dots\be_r},
 \end{array}
 $$
therefore $b=c_0=c^j=0$, for all $j=1,\dots,m$, $d_j=b_i=0$, for all
$j=2,\dots,r,\; i=1,\dots, r-1$ and $b_r=d_1$. Using this we
conclude that
$$\begin{array}{rcl}
    \de\bigl(e_{m+n+1}\bigr) & = &\displaystyle{ b\be_1\dots\be_r\ot e_{m+n+1}}, \\
    \de\bigl(e_0\bigr) & = & \displaystyle{be_0\ot\be_1\dots\be_r+\sum_{j=1}^nc_j\al_{m+1}\dots\al_{m+j}\ot e_0} \\
     & + & \displaystyle{ \sum_{j=1}^n\sum_{k=1}^mc_{jk}\al_{m+1}\dots\al_{m+j}\ot\al_k\dots\al_m +\sum_{j=1}^n\sum_{k=1}^rd_{jk}\al_{m+1}\dots\al_{m+j}\ot\be_k\dots\be_r
     }.
  \end{array}
$$
The next step is to consider the path $\al_{m+1}\dots\al_{m+n}$, for
this path the coproduct is determined by
$$\de\bigl(\al_{m+1}\dots\al_{m+n}\bigr)=\bigl(\al_{m+1}\dots\al_{m+n}\ot 1\bigr)\de\bigl(e_{m+n}\bigr)=a\al_{m+1}\dots\al_{m+n}\ot\al_1\dots\al_{m+n}.$$
By the other hand, this coproduct is determined by\\
$\displaystyle{\de\bigl(\al_{m+1}\dots\al_{m+n}\bigr)=\de\bigl(e_0\bigr)\bigl(1\ot\al_{m+1}\dots\al_{m+n}\bigr)=\sum_{j=1}^nc_j\al_{m+1}\dots\al_{m+j}\ot\al_{m+1}\dots\al_{m+n}}$\\
$\displaystyle{+\sum_{j=1}^n\sum_{k=1}^mc_{jk}
\al_{m+1}\dots\al_{m+j}\ot\al_k\dots\al_{m+n}}$.\\ Comparing the
expressions we have $c_j=0$, for all $j=1,\dots,n$, $c_{jk}=0$, for
all $j=1,\dots, n-1$, $i=2,\dots,m$ and $a=c_{n1}$. Consequently
$$\de\bigl(e_0\bigr)=a\al_{m+1}\dots\al_{m+n}\ot\al_1\dots\al_m+be_0\ot\be_1\dots\be_r+\sum_{j=1}^n\sum_{k=1}^rd_{jk}\al_{m+1}\dots\al_{m+j}\ot\be_k\dots\be_r.$$
Finally, we consider the path $\al_1\dots\al_m$, as before, we have
two way to define the coproduct in this path
$$\begin{array}{rcl}
    \de\bigl(\al_1\dots\al_m\bigr) & = &\displaystyle{ \de\bigl(e_1\bigr)\bigl(1\ot\al_1\dots\al_m\bigr)=a\al_1\dots\al_{m+n}\ot\al_1\dots\al_m}\\
                                   & = &\displaystyle{ \bigl(\al_1\dots\al_m\ot 1\bigr)\de\bigl(e_0\bigr)= a\al_1\dots\al_{m+n}\ot\al_1\dots\al_m}\\
                                   & + &\displaystyle{b\al_1\dots\al_m\ot\be_1\dots\be_r+\sum_{j=1}^n\sum_{k=1}^rd_{jk}\al_{1}\dots\al_{m+j}\ot\be_k\dots\be_r, }
\end{array}$$
then $b=d_{jk}=0$ for all $j,k$.\\
As $b=0$ we have that $\de\bigl(e_{m+n+1}\bigr)=0$, this implies
that $\de\bigl(e_{m+n+i}\bigr)=0$ for all $i=1,\dots,r$ and
$\displaystyle{\de\bigl(e_0\bigr)=a\al_{m+1}\dots\al_{m+n}\ot\al_1\dots\al_m}$.

It is a simple calculation to prove that $\de$ is coassociative.
\end{proof}
The next result is the symmetrical case to Lemma \ref{lemma1}.
\begin{lem}
The algebra associated to the quiver
$$\scalebox{.8}{\includegraphics{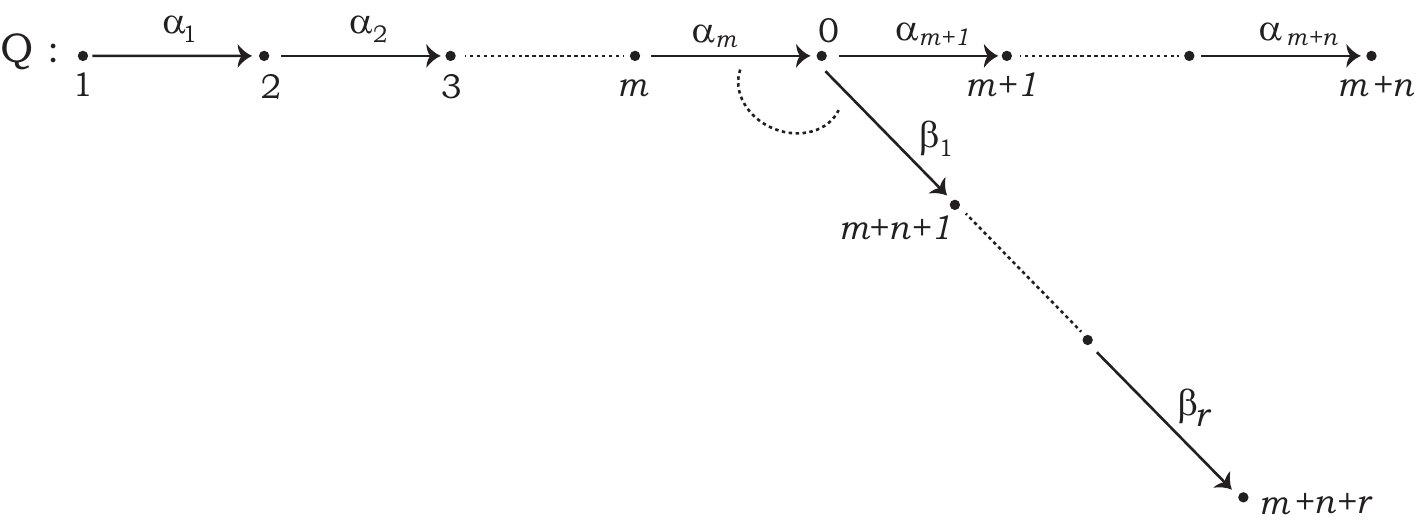}}$$
with the relation $\al_{m}\be_1=0$, admits only one nearly Frobenius
structure, this is
$$\begin{array}{rclcrcl}
        \de\bigl(e_1\bigr) & = & \displaystyle{\al_1\dots\al_{m+n}\ot e_1,} &  & \de\bigl(\al_1\bigr) & = & \displaystyle{\al_1\dots\al_{m+n}\ot\al_1,} \\
        \de\bigl(e_i\bigr) & = & \displaystyle{\al_i\dots\al_{m+n}\ot\al_1\dots\al_{i-1}}, &  & \de\bigl(\al_i\bigr) & = & \displaystyle{\al_i\dots\al_{m+n}\ot\al_1\dots\al_i,}\\
        \de\bigl(e_{m+n}\bigr) & = & \displaystyle{e_{m+n}\ot\al_1\dots\al_{m+n},} &  & \de\bigl(\al_{m+n}\bigr) & = & \displaystyle{\al_{m+n}\ot\al_1\dots\al_{m+n},} \\
  \end{array}$$
$$\de\bigl(e_{m+n+i}\bigr) =  0 \quad\mbox{and}\quad \de\bigl(\be_i\bigr)  =  0\quad\mbox{for all}\; i=1,\dots,r.$$
Accordingly $\op{Frobdim}(A)=1$.
\end{lem}
\begin{lem}
The algebra associated to the quiver
$$\scalebox{0.8}{\includegraphics{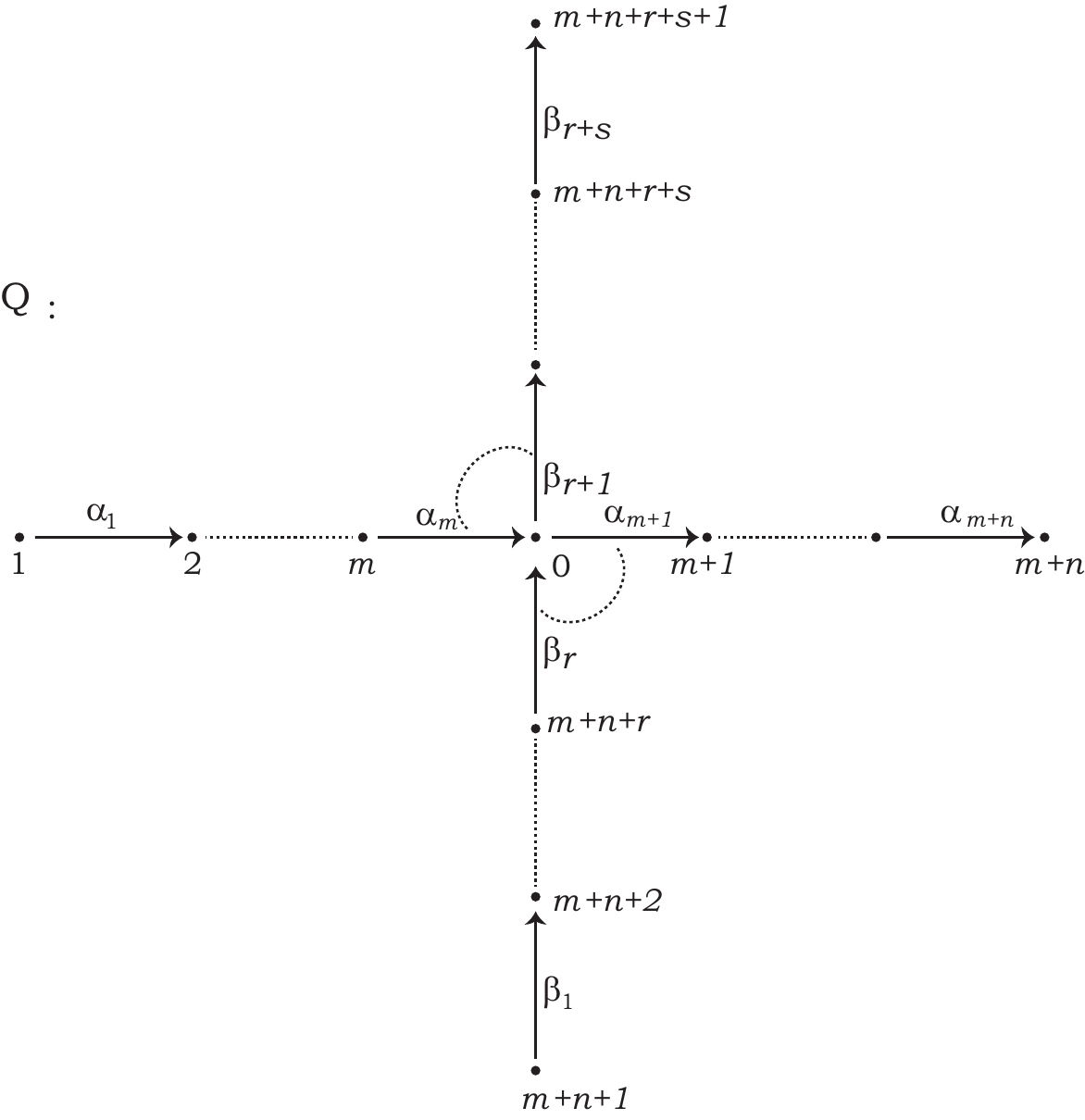}}$$
with relations $\be_r\al_{m+1}=0$ and $\al_m\be_{r+1}=0$,admits two
independent nearly Frobenius structures and a general coproduct is
determined by
$$\begin{array}{rclcrcl}
\displaystyle{ \de\bigl(e_1\bigr)} & = & \displaystyle{a\al_1\dots\al_{m+n}\ot e_1,} &  & \de\bigl(\al_1\bigr) & = & \displaystyle{a\al_1\dots\al_{m+n}\ot\al_1,} \\
\displaystyle{\de\bigl(e_i\bigr)} & = & \displaystyle{a\al_i\dots\al_n\ot\al_1\dots\al_{i-1},} &  & \de\bigl(\al_i\bigr) & = & \displaystyle{a\al_i\dots\al_n\ot\al_1\dots\al_i,} \\ \displaystyle{\de\bigl(e_{m+n}\bigr)} & = & \displaystyle{ae_{m+n}\ot\al_1\dots\al_{m+n},} &  & \de\bigl(\al_{m+n}\bigr) & = & \displaystyle{a\al_{m+n}\ot\al_1\dots\al_{m+n},} \\
\displaystyle{\de\bigl(e_{m+n+1}\bigr)} & = & \displaystyle{b\be_1\dots\be_{r+s}n\ot e_1,} &  & \de\bigl(\be_1\bigr) & = & \displaystyle{b\be_1\dots\be_{r+s}\ot\be_1,} \\
\displaystyle{\de\bigl(e_{m+n+i}\bigr)} & = & \displaystyle{b\be_i\dots\be_{r+s}\ot\be_1\dots\be_{i-1},} &  & \de\bigl(\be_i\bigr) & = & \displaystyle{b\be_i\dots\be_{r+s}\ot\be_1\dots\be_i,}\\
\displaystyle{\de\bigl(e_{m+n+r+s+1}\bigr)} & = &
\displaystyle{be_{m+n+r+s+1}\ot\be_1\dots\be_{r+s},} &  &
\de\bigl(\be_{r+s}\bigr) & =
&\displaystyle{b\be_{r+s}\ot\be_1\dots\be_{r+s},}
\end{array}
$$
$$\de\bigl(e_0\bigr) = a\al_{m+1}\dots\al_{m+n}\ot\al_1\dots\al_m +b \be_{r+1}\dots\be_{r+s}\ot\be_{1}\dots\be_r,$$
where $a,b\in\k$. In this case $\op{Frobdim}(A)=2$.
\end{lem}
\begin{proof}
First, note that, we can expressed the coproduct in the vertices as
$$\begin{array}{rcl}
     \de\bigl(e_1\bigr) & = & \displaystyle{a^1e_1\ot e_1+\sum_{i=1}^{m+n}a_i\al_1\dots\al_i\ot e_1}, \\
     \de\bigl(e_{m+n}\bigr) & = & \displaystyle{b^ne_{m+n}\ot
     e_{m+n}+\sum_{i=1}^{m+n}b_ie_{m+n}\ot\al_i\dots\al_{m+n}}.
  \end{array}
$$
Applying the coproduct in the path $\al_1\dots\al_{m+n}$ we can
prove that
 $$\begin{array}{rcl}
     \de\bigl(e_1\bigr) & = & \displaystyle{a\al_1\dots\al_{m+n}\ot e_1}, \\
     \de\bigl(e_{m+n}\bigr) & = & \displaystyle{ae_{m+n}\ot
     \al_1\dots\al_{m+n}}.
  \end{array}
$$
By symmetry we have that
 $$\begin{array}{rcl}
     \de\bigl(e_{m+n+1}\bigr) & = & \displaystyle{b\be_1\dots\be_{r+s}\ot e_{m+n+1}}, \\
     \de\bigl(e_{m+n+r+s+1}\bigr) & = & \displaystyle{be_{m+n+r+s+1}\ot
     \be_1\dots\be_{r+s}}.
  \end{array}
$$
Reproducing the calculus of Lemma \ref{lemma1} we can prove that
$$\begin{array}{rcl}
    \de\bigl(e_i\bigr) & = & \displaystyle{a\al_i\dots\al_{m+n}\ot \al_1\dots\al_{i-1}}, \\
     \de\bigl(e_{m+n+i}\bigr) & = & \displaystyle{a\be_{i}\dots\be_{r+s}\ot
     \be_1\dots\be_{i-1}}.
  \end{array}$$
In the vertex $e_0$ the situation is more complicated.
$$\begin{array}{rcl}
    \de\bigl(e_0\bigr) & = & \displaystyle{b_0e_0\ot e_0+\sum_{i=1}^mb_ie_0\ot\al_i\dots\al_{m}+\sum_{j=1}^rc_je_0\ot\be_j\dots\be_r}\\
                       & + & \displaystyle{\sum_{i=1}^{n}b^i\al_{m+1}\dots\al_{m+i}\ot e_0+\sum_{j=1}^sc^j\be_{r+1}\dots\be_{r+j}\ot e_0 }\\
                       & + & \displaystyle{\sum_{i=1}^n\sum_{j=1}^mb_{ij}\al_{m+1}\dots\al_{m+i}\ot\al_j\dots\al_m+\sum_{i=1}^s\sum_{j=1}^rc_{ij}\be_{r+1}\dots\be_{r+i}
                       \ot\be_{j}\dots\be_{r} }\\
                       & + & \displaystyle{\sum_{i=1}^n\sum_{j=1}^rb^{ij}\al_{m+1}\dots\al_{m+i}\ot\be_{j}\dots\be_{r}+\sum_{i=1}^s\sum_{j=1}^mc^{ij}\be_{r+1}\dots\be_{r+i}
                       \ot\al_j\dots\al_m}. \end{array}$$
If we determine the coproduct in the paths $\al_1\dots\al_m$ and
$\be_1\dots\be_r$ we conclude that
$$\de\bigl(e_0\bigr)=a\al_{m+1}\dots\al_{m+n}\ot\al_1\dots\al_m+b\be_{r+1}\dots\be_{r+s}\ot\be_1\dots\be_r.$$
The coproduct in the arrows is determined by the value in the
vertices.

The coassociativity is an easy exercise.
\end{proof}

Now, if we consider a gentle algebra $A$ associated to $Q$, a finite
connected quiver without oriented cycles, we can produce an
algorithm that permit us to determine the number independent nearly
Frobenius structures that the algebra $A$ admits. Next we develop
the algorithm.

As $Q$ is finite we can suppose that $\# Q_0=n$ and $\# Q_1=m$. The
quiver $Q$ is triangular, because it has not cycles. In particular,
there exist a partial order $\preccurlyeq$ of
$\bigl\{1,2,\dots,n\bigr\}$ and the arrows such that
$$\left\{\begin{array}{lcl}
                                                         i\prec j& \mbox{if} & i \rightsquigarrow j\; \bigl(\mbox{$i\neq j$, $i$ precede to $j$}\bigr) \\
                                                         i \preccurlyeq j & \mbox{if} & i=j\;\mbox{or}\; i\prec j
                                                       \end{array}
\right.$$ and
$$\left\{\begin{array}{lcl}
                                                         \al\prec \be& \mbox{if} & t\bigl(\al\bigr) \preccurlyeq\bigl(\be\bigr), \\
                                                         \al \preccurlyeq \be & \mbox{if} & \al=\be\;\mbox{or}\; \al\prec
                                                         \be
                                                       \end{array}
\right.$$

Let $\ma{F}=\bigl\{\mbox{sources of}\; Q\bigr\}$. Note that the
sources of the quiver $Q$ are the minimal elements with the order
$\preccurlyeq$. In addition, any vertex of the quiver $Q$ is a
source or one of the following
$$\scalebox{.9}{\includegraphics{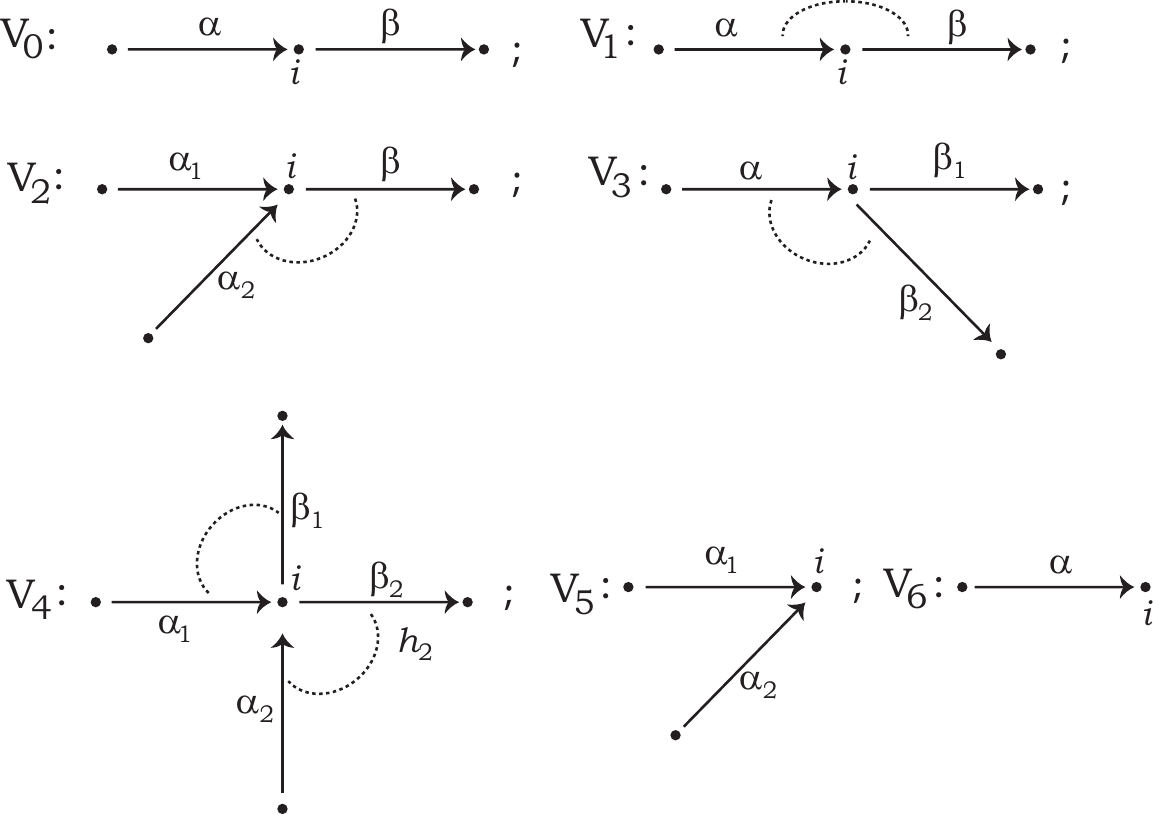}}$$
$$Q_0=\ma{F}\cup V_0\cup V_1\cup V_2\cup V_3\cup V_4\cup V_5\cup V_6.$$
This decomposition permit us to define the \emph{type} of a vertex
as:
$$\op{type}(i):=j\quad \mbox{if}\quad i\in V_j,\; j=0,\dots ,6.$$
With the previous order we can define vectors  associated to the
vertexes and associated to the arrows, these are
$y=\bigl(y_1,\dots,y_n\bigr)\in\mb{Z}^n$ and
$x=\bigl(x_1,\dots,x_m\bigr)\in\mb{Z}^m$. In the next paragraph we
describe the process of construction.

Step 0: Let $f\in\ma{F}$ and $\op{gr}_s(f)=1$, that is there exists a unique $\be\in Q_1$, with $s\bigl(\be\bigr)=f$, we define $y_f=1$ and $x_\be=1$.\\
If $f\in\ma{F}$ and $\op{gr}_s(f)=2$, that is there exist $\be_1, \be_2\in Q_1$, with $s\bigl(\be_1\bigr)=s\bigl(\be_2\bigr)=f$, we define $y_f=0$ and $x_{\be_1}=x_{\be_2}=0$.\\
We introduce a counter $d\in\mb{N}$ starting in $0$.

Step 1: Let $U=Q_0-\ma{F}$ and $i\in U$ minimal.
\begin{itemize}
  \item If $\op{type}(i)=0$ we define $y_i=x_\al$, $x_\be=y_i$ and the new set is $U:=U-\{i\}$. 
  \item If $\op{type}(i)=1$ we define $y_i=l_l(i)l_r(i)+2+x_{\al}-1$, $x_\be=y_i$, where $l_r(i)=\op{max}\bigl\{\op{long}(w):\; w\;\mbox{path}\;, s(w)=i\bigr\}$
and $l_l(i)=\op{max}\bigl\{\op{long}(w):\; w\;\mbox{path}\;,
t(w)=i\bigr\}$ and $U:= U-\{i\}$.
  \item If $\op{type}(i)=2$ we define $y_i=x_{\al_1}$, $x_\be=y_i$ and $U:= U-\{i\}$,
  $$d:=x_{\al_2}-1+d.$$
  \item If $\op{type}(i)=3$ we define $y_i=x_\al$, $x_{\be_2}=0$, $x_{\be_1}=x_{\al}$ and $U:= U-\{i\}$.
  \item If $\op{type}(i)=4$ we define $y_i=y_{j_1}+y_{j_2}$, $x_{\be_2}=x_{\al_1}$, $x_{\be_1}=x_{\al_2}$ and $U:=U-\{i\}$.
  \item If $\op{type}(i)=5$ we define $y_i=0$ and
  $d:=d+x_{\al_1}+x_{\al_2}-\delta\bigl(\al_1\bigr)-\delta\bigl(\al_2\bigr)$,
  where $\delta:Q_1\rt\mb{Z}$ is defined by $\delta(\al)=0$ if
  $x_{\al}=0$ and $\delta(\al)=1$ if $x_\al\geq 1$.
  \item If $\op{type}(i)=6$ we define $y_i=x_\al$ and $d=d+x_\al$.
\end{itemize}
We repeat this process recursiveness over the set $U$, and finally
the number $d$ is the total of nearly Frobenius structures that the
algebra admits.

\begin{cor}
Let $A$ be a gentle algebra associated to $Q$, a finite connected
quiver without oriented cycles. Then, $A$ has finite Frobenius
dimension
\end{cor}

We apply the previous algorithm in the next two examples.
\begin{example}
The gentle algebra $A$ associated to the quiver
$$\scalebox{.8}{\includegraphics{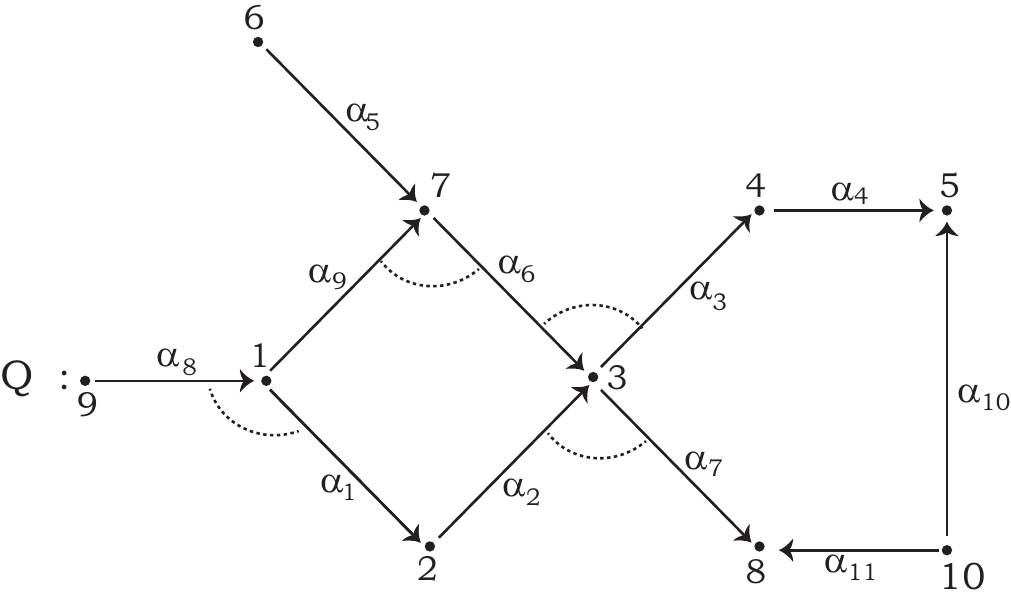}}$$ with relations $\al_9\al_6=0$, $\al_8\al_1=0$, $\al_2\al_7=0$ and $\al_6\al_3=0$, has $\op{Frobdim}(A)=0$, that is the only nearly Frobenius structure that this algebra admits is the trivial ($\de\equiv 0$).

Applying the algorithm we have the next situation
$$\scalebox{.8}{\includegraphics{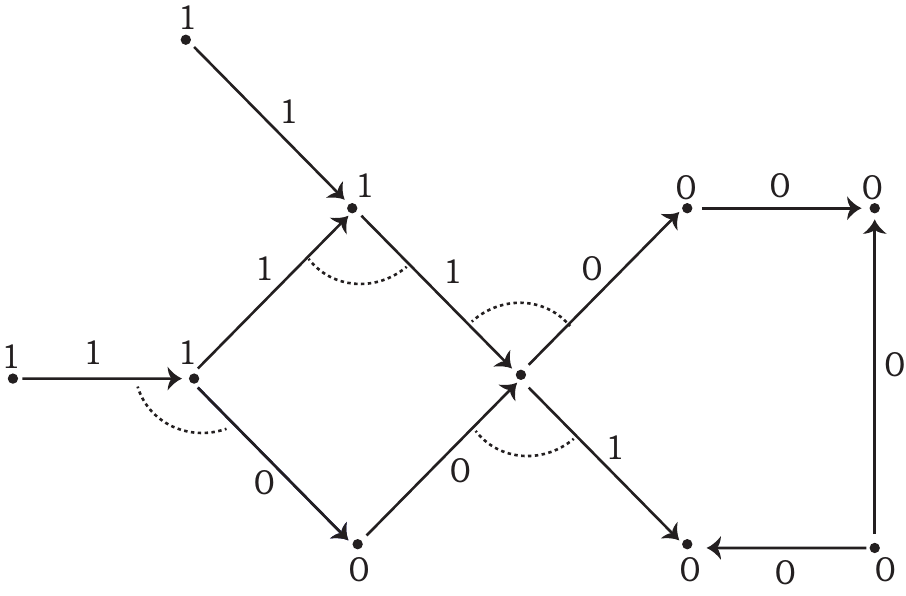}}$$ and the counter $d$ is zero. Then, in this example, we have only the trivial nearly Frobenius coproduct.
\end{example}

\begin{example}
If we consider the algebra $A$ associated to the quiver
$$\scalebox{.8}{\includegraphics{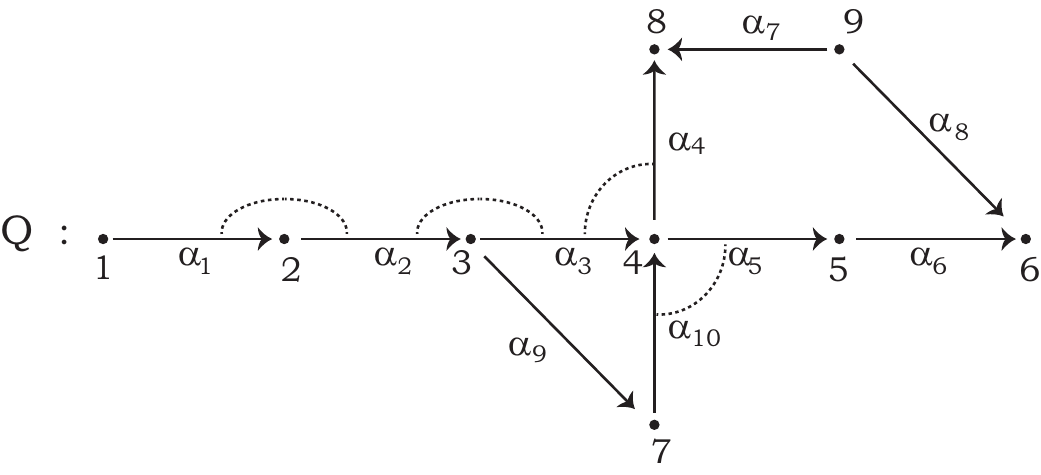}}$$ with relations $\al_1\al_2=0$, $\al_2\al_3=0$, $\al_3\al_4=0$ and $\al_{10}\al_5=0$. It is possible to determine all the nearly Frobenius structures, they are
$$\begin{array}{rcl}
    \de(e_1) & = & a_1\al_1\ot e_1 \\
    \de(\al_1) & = & a_1\al_1\ot\al_1 \\
    \de(e_2) & = & a_1 e_2\ot\al_1 +a_2\al_2\ot\al_1+a_3\al_2\al_9\ot\al_1+a_4\al_2\al_9\al_{10}\ot\al_1+a_5\al_2\al_9\al_{10}\al_4\ot\al_1, \\
    \de & = & 0\quad\mbox{on the other cases}
  \end{array}
$$ In this case $\op{Frobdim}(A)=5$. We determine, applying the algorithm, that the counter $d$ is five and we conclude that $\op{Frobdim}(A)=5$. In the next diagram we represent
the vectors given in the algorithm associated to the vertex and
arrows
$$\scalebox{.8}{\includegraphics{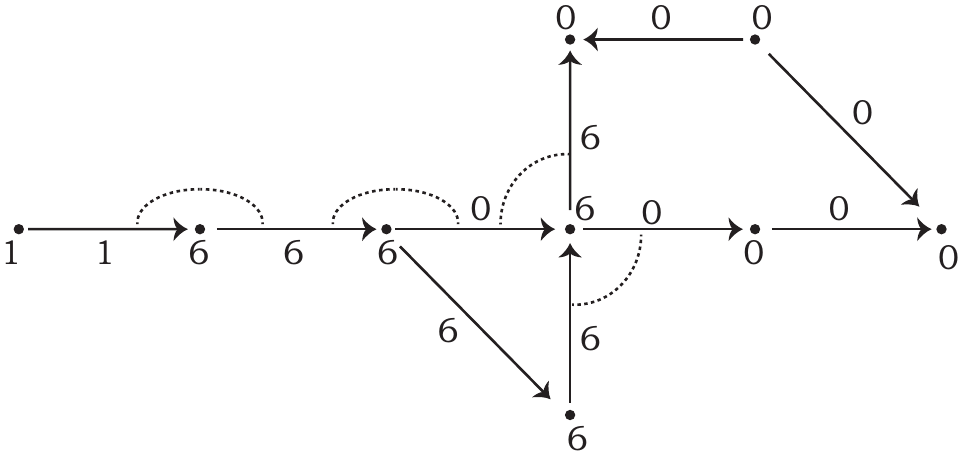}}$$
 then $d=5$.
\end{example}


\subsection{Comparing dimensions}
\hspace{.5cm} In this subsection we determine the Frobenius
dimension of algebras associated to cyclic quivers. Using this
result we exhibit a family of algebras with Frobenius dimension
great that it's dimension over $\k$.

Let $C\bigl(n_1,n_2,\dots,n_m\bigr)$ the quiver
$$\includegraphics{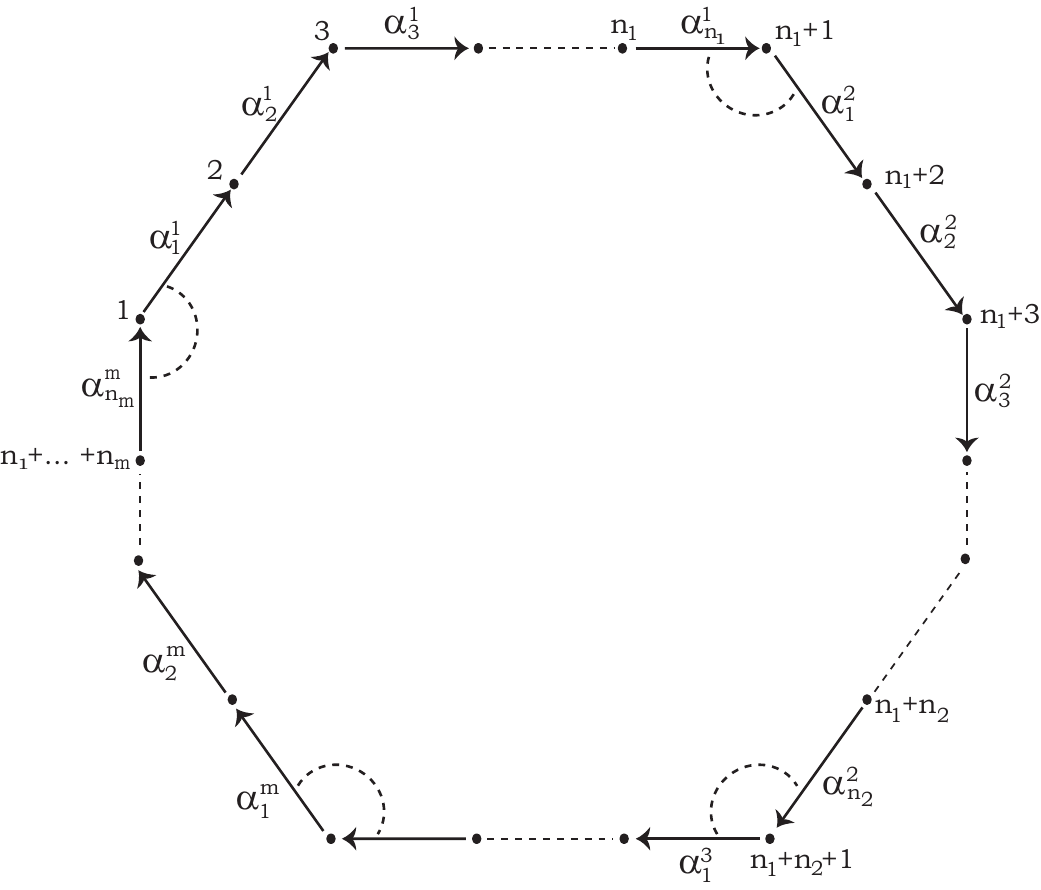}$$
where $m,n_1,n_2,\dots,n_m\in\mb{N}^*$ and
$\displaystyle{A_C=\frac{\k C}{I_C}}$ with $$I_C=\bigl\langle
\al_{n_m}^m\al_1^1,\; \al_{n_i}^i\al_1^{i+1},\; i=1,\dots,
m-1\bigr\rangle.$$

\begin{thm}\label{theorem2}
$\op{Frobdim}(A)=m+\sum_{i=1}^mn_in_{i+1}$, with $n_{m+1}=n_1.$
\end{thm}
\begin{proof}
We will determine the coproduct in the vertices of the algebra $A$
using the formula
$$\de\bigl(e_i\bigr)=\de\bigl(e_i\bigr)\bigl(1\ot
e_i\bigr)=\bigl(e_i\ot 1\bigr)\de\bigl(e_i\bigr).$$
\begin{itemize}
  \item If $i=1$ then
  $$\begin{array}{rcl}
      \de\bigl(e_1\bigr) & = & \displaystyle{a_1e_1\ot e_1+\sum_{k=1}^{n_m}a_1^ke_1\ot\al_k^m\dots\al_{n_m}^m+\sum_{j=1}^{n_1}a_j^1\al_1^1\dots\al_j^1\ot e_1 }\\
       & + & \displaystyle{\sum_{j=1}^{n_1}\sum_{k=1}^{n_m}a_{jk}^1\al_1^1\dots\al_j^1\ot\al_k^m\dots\al_{n_m}^m}.
           \end{array}
$$
  \item If the index is $n_1+\dots+n_i+1$, $i=1,\dots,m-1$ then
  $$\begin{array}{rcl}
      \de\bigl(e_{n_1+\dots+n_i+1}\bigr) & = & \displaystyle{a_ie_{n_1+\dots+n_i+1}\ot e_{n_1+\dots+n_i+1}+\sum_{k=1}^{n_i}a_k^ie_{n_1+\dots+n_i+1}\ot\al_k^i\dots\al_{n_i}^i }\\
       & + & \displaystyle{\sum_{j=1}^{n_{i+1}}a_i^j\al_1^{i+1}\dots\al_j^{i+1}\ot e_{n_1+\dots+n_i+1}+\sum_{j=1}^{n_{i+1}}\sum_{k=1}^{n_i}a_{jk}^i\al_1^{i+1}\dots\al_j^{i+1}\ot\al_k^i\dots\al_{n_i}^i}.
    \end{array}$$
  \item If $k=2,\dots n_{i}$ then
   $$\begin{array}{rcl}
      \de\bigl(e_{n_1+\dots n_{i-1}+k}\bigr) & = & \displaystyle{b_k^ie_{n_1+\dots n_{i-1}+k}\ot e_{n_1+\dots n_{i-1}+k}+\sum_{l=1}^{k-1}b_l^{ik}e_{n_1+\dots n_{i-1}+k}\ot\al_l^i\dots\al_{k-1}^i}\\
      & + & \displaystyle{\sum_{j=k}^{n_i}b_k^{ij}\al_k^i\dots\al_j^i\ot e_{n_1+\dots n_{i-1}+k} +\sum_{j=k}^{n_i}\sum_{l=1}^{k-1}b_{jl}^{ik}\al_k^i\dots\al_j^i\ot\al_l^i\dots\al_{k-1}^i}.
          \end{array}
  $$
Now, we consider the maximal path $\al_1^i\dots \al_{n_i}^i$, for
this the coproduct satisfies
\begin{equation}\label{equa1}
\de\bigl(\al_1^i\dots\al_{n_i}^i\bigr)=\de\bigl(e_{n_1+\dots
n_{i-1}+1}\bigr)\bigl(1\ot\al_1^i\dots\al_{n_i}^i\bigr)=\bigl(\al_1^i\dots\al_{n_i}^i\ot
1\bigr)\de\bigl(e_{n_1+\dots n_{i}+1}\bigr)\end{equation}

By substitution in (\ref{equa1}) we get:
$$\begin{array}{c}
    \displaystyle{a_{i-1}e_{n_1+\dots+n_{i-1}+1}\ot\al_1^i\dots\al_{n_i}^i+\sum_{j=1}^{n_{i}}a_{i-1}^j\al_1^{i}\dots\al_j^{i}\ot \al_1^i\dots\al_{n_i}^i }\\
    = \\
    \displaystyle{a_i\al_1^i\dots\al_{n_i}^i\ot e_{n_1+\dots+n_i+1}+\sum_{k=1}^{n_i}a_k^i\al_1^i\dots\al_{n_i}^i\ot\al_k^i\dots\al_{n_i}^i }
  \end{array}
$$
then $a_i=a_{i-1}=0$, $a_k^i=0$ $\forall\; k=2,\dots ,n_i$,
$a_{i-1}^j=0$ $\forall\; j=1,\dots,n_i-1$ and $a_1^i=a_{i-1}^{n_i}.$

Therefore
 $$\begin{array}{rcl}
      \de\bigl(e_{n_1+\dots+n_i+1}\bigr) & = & \displaystyle{a_ie_{n_1+\dots+n_i+1}\ot\al_1^i\dots\al_{n_i}^i+a_{i+1}\al_1^{i+1}\dots\al_{n_{i+1}}^{i+1}\ot e_{n_1+\dots+n_i+1} }\\
       & + & \displaystyle{\sum_{j=1}^{n_{i+1}}\sum_{k=1}^{n_i}a_{jk}^i\al_1^{i+1}\dots\al_j^{i+1}\ot\al_k^i\dots\al_{n_i}^i}
    \end{array}$$
If we study the particular paths $\al_1^i\dots\al_{k-1}^i$ we can determine the relation between the coproduct values in the vertices $e_{n_1+\dots +n_{i-1}+1}$ and $e_{n_1+\dots+n_{i-1}+k}$.%

The coproduct in $\al_1^i\dots\al_{k-1}^i$ satisfies
$$\de\bigl(\al_1^i\dots\al_{k-1}^i\bigr)=\de\bigl(e_{n_1+\dots +n_{i-1}+1}\bigr)\bigl(1\ot\al_1^i\dots\al_{k-1}^i\bigr)=\bigl(\al_1^i\dots\al_{k-1}^i\ot 1\bigr)\de\bigl(e_{n_1+\dots+n_{i-1}+k}\bigr)$$
then
$$\begin{array}{c}
     \displaystyle{a_{i}\al_1^i\dots\al_{n_i}^i\ot \al_1^i\dots\al_{k-1}^i  }\\
     =\\
\displaystyle{b_k^i\al_1^i\dots\al_{k-1}^i\ot e_{n_1+\dots n_{i-1}+k}+\sum_{l=1}^{k-1}b_l^{ik}\al_1^i\dots\al_{k-1}^i\ot\al_l^i\dots\al_{k-1}^i}\\
+ \displaystyle{\sum_{j=k}^{n_i}b_k^{ij}\al_1^i\dots\al_j^i\ot
e_{n_1+\dots n_i+k}
+\sum_{j=k}^{n_i}\sum_{l=1}^{k-1}b_{jl}^{ik}\al_1^i\dots\al_j^i\ot\al_l^i\dots\al_{k-1}^i}
    \end{array}
$$
If we compare the expressions  we conclude that $b_k^i=b_k^{ij}=b_l^{ik}=0$ $\forall\;j,l$, $b_{jl}^{ik}=0$ $\forall\; j=k,\dots, n_i-1$, $l=2,\dots,k-1$ and $b_{n_i1}^{ik}=a_i$\\
Consequently
$$\de\bigl(e_{n_1+\dots+n_{i-1}+k}\bigr)=a_i\al_k^i\dots\al_{n_i}^i\ot\al_1^i\dots\al_{k-1}^i.$$

The coassociativity is satisfied by a simple calculus. Then,
Counting the independent coefficients we determine that $A$ admits
$m+\sum_{i=1}^mn_in_{i+1}$ independent nearly-Frobenius structures.
\end{itemize}
\end{proof}

\begin{cor}
If $n_1=n_2=\dots=n_m=t$, with $t\geq 3$, then $\op{Frobdim}(A)>
\op{dim}_\k(A)$.
\end{cor}
\begin{proof}
In this case the dimension of $A$, as vector space, is
$$\op{dim}_\k(A)=\frac{m\bigl(t^2+3t\bigr)}{2}$$
and $\op{Frobdim}(A)=m\bigl(1+t^2\bigr)$. If we compare these
expressions we conclude that $\op{Frobdim}(A)> \op{dim}_\k(A)$ if
$t>2$.
\end{proof}


\begin{thebibliography}{10} 
\addtolength{\leftmargin}{0.2in} 
\setlength{\itemindent}{-0.2in}
\bibitem[ASS06]{ASS06}
I. Assem, D. Simson,  and A. Skowronski, \emph{Elements of the
{R}epresentation {T}heory of {A}ssociative {A}lgebras {V}olume 1
{T}echniques of {R}epresentation {T}heory,} Cambridge {U}niversity
{P}ress, 2006.
\bibitem[CG04]{cohen} Ralph L. Cohen and Veronique Godin,
\emph{A polarized view of string topology}, Topology, geometry, and
quantum field theory , Cambridge: Cambridge University Press. London
Mathematical Society Lecture Notes \textbf{308} (2004), 127-154.
\bibitem[Gab72]{Gabriel72} P. Gabriel, \emph{Unzerlegbare {D}arstellungen
{I}}, Manuscripta Math. \textbf{6} (1972), 71-103.
\bibitem[Gab73]{Gabriel73} P. Gabriel, \emph{Indecomposable representations {II}},
Symposia Mat. Inst. Naz. Alta Mat.\textbf{11} (1973), 81-104.
\bibitem[GLSU13]{AnaErnesto} A. Gonz\'alez, E. Lupercio, C. Segovia, and B. Uribe,
\emph{{O}rbifold {T}opological {Q}uantum {F}ield {T}heories in
{D}imension 2}, Book finished, 2013.
\bibitem[L\'ev04]{Jessica} Jessica L{\'e}vesque, \emph{Produits Fibr\'es d'algebres et inclinaison},
Ph.D. thesis, Facult\'e des Sciencies, Universit\'e de Sherbrooke,
2004.
\end{thebibliography}
\end{document}